\documentclass[smallextended, final]{svjour3}
\usepackage{geometry}
\usepackage{graphicx}
\usepackage{amssymb,amsmath}
\usepackage{epstopdf}
\usepackage{subfigure}
\usepackage[authoryear,sort&compress]{natbib}
\bibpunct{[}{]}{;}{n}{,}{,}
\usepackage[colorlinks=false,bookmarksdepth=0]{hyperref}
\usepackage{float,enumerate}

\newcommand{\bfi}{\bfseries\itshape}
\DeclareMathOperator*{\ext}{ext}

\DeclareMathOperator{\ad}{ad}
\DeclareMathOperator{\Ad}{Ad}
\DeclareMathOperator{\dexp}{dexp}
\DeclareMathOperator{\ddexp}{ddexp}

\DeclareMathOperator{\SO}{SO(3)}

\newcommand{\so}{\ensuremath{\mathfrak{so}(3)}}

\newcommand{\tr}[1]{\mbox{tr}\ensuremath{\negthickspace\bracket{#1}}}
\newcommand{\norm}[1]{\ensuremath{\left\| #1 \right\|}}
\newcommand{\bracket}[1]{\ensuremath{\left[ #1 \right]}}
\newcommand{\braces}[1]{\ensuremath{\left\{ #1 \right\}}}
\newcommand{\parenth}[1]{\ensuremath{\left( #1 \right)}}

\newcommand{\refeqn}[1]{(\ref{eqn:#1})}

\restylefloat{figure}

\journalname{Frontiers of Mathematics in China}
\date{\today}
\title{General Techniques for Constructing Variational Integrators}
\author{Melvin Leok \and Tatiana Shingel}
\institute{Department of Mathematics, University of California, San Diego.
\email{\{mleok,tshingel\}@math.ucsd.edu}}

\allowdisplaybreaks
\begin{document}
\maketitle
\begin{abstract}The numerical analysis of variational integrators relies on variational error analysis, which relates the order of accuracy of a variational integrator with the order of approximation of the exact discrete Lagrangian by a computable discrete Lagrangian. The exact discrete Lagrangian can either be characterized variationally, or in terms of Jacobi's solution of the Hamilton--Jacobi equation. These two characterizations lead to the Galerkin and shooting-based constructions for discrete Lagrangians, which depend on a choice of a numerical quadrature formula, together with either a finite-dimensional function space or a one-step method. We prove that the properties of the quadrature formula, finite-dimensional function space, and underlying one-step method determine the order of accuracy and momentum-conservation properties of the associated variational integrators. We also illustrate these systematic methods for constructing variational integrators with numerical examples.

\end{abstract}

\section{Introduction}
Variational integrators are a particularly popular class of geometric numerical integrators, due to their symplectic and momentum conservation properties, and their good energy behavior. While most of the standard approaches for constructing variational integrators involve piecewise polynomial interpolants, and high-order quadrature formulas, we will show that existing techniques from approximation theory, numerical quadrature, and one-step methods, can be incorporated systematically into the variational integration framework.

In this paper we will discuss two general approaches to constructing variational integrators. The first of this is the Galerkin variational integrator construction, which relies on a variational characterization of the exact discrete Lagrangian, and involves the choice of a numerical quadrature formula, and a finite-dimensional function space. The second approach relies on the characterization of the exact discrete Lagrangian in terms of Jacobi's solution of the Hamilton--Jacobi equation, and involves the choice of a numerical quadrature formula, and an underlying one-step method.

\subsection{Discrete Lagrangian Mechanics}\index{discrete mechanics}
Discrete Lagrangian mechanics~\citep{MaWe2001} is based on a discrete analogue of Hamilton's principle, referred to as  the {\bfi discrete Hamilton's principle},
\[ \delta \mathbb{S}_d = 0,\]
where the {\bfi discrete action sum}, $\mathbb{S}_d:Q^{n+1}\rightarrow \mathbb{R}$, is given by
\[ \mathbb{S}_d(q_0,q_1,\ldots,q_n) = \sum\nolimits_{i=0}^{n-1} L_d(q_i, q_{i+1}) .\]
The {\bfi discrete Lagrangian}, $L_d:Q\times Q\rightarrow \mathbb{R}$, is a generating function of the symplectic flow, and is an approximation to the {\bfi exact discrete Lagrangian},
\begin{equation}
L_d^E(q_0,q_1;h)=\int_0^h L(q_{01}(t),\dot q_{01}(t)) dt,\label{exact_Ld_Jacobi}
\end{equation}
where $q_{01}(0)=q_0,$ $q_{01}(h)=q_1,$ and $q_{01}$ satisfies the
Euler--Lagrange equation in the time interval $(0,h)$. The exact discrete Lagrangian is related to the Jacobi solution of the Hamilton--Jacobi equation. Alternatively, one can characterize the exact discrete Lagrangian in the following way,
\begin{equation}
L_d^E(q_0,q_1;h)=\ext_{\substack{q\in C^2([0,h],Q) \\ q(0)=q_0, q(h)=q_1}} \int_0^h L(q(t), \dot q(t)) dt.\label{exact_Ld_variational}
\end{equation}
As we will see in the remainder of the paper, these two characterizations of the exact discrete Lagrangian will lead to two general techniques for constructing variational integrators.

The discrete variational principle then yields the {\bfi discrete Euler--Lagrange (DEL)} equation,
\begin{equation}\label{DEL}
 D_2 L_d(q_{k-1},q_k)+D_1 L_d(q_k,q_{k+1})=0,
\end{equation}
which implicitly defines the {\bfi discrete Lagrangian map} $F_{L_d}:(q_{k-1},q_k)\mapsto(q_k,q_{k+1})$ for initial conditions $(q_0,q_1)$ that are sufficiently close to the diagonal of $Q\times Q$. This is equivalent to the {\bfi implicit discrete Euler--Lagrange (IDEL)} equations,
\begin{equation}
p_k=-D_1 L_d(q_k, q_{k+1}),\qquad p_{k+1}=D_2 L_d(q_k, q_{k+1}),\label{IDEL}
\end{equation}
which implicitly defines the {\bfi discrete Hamiltonian map} $\tilde{F}_{L_d}:(q_k,p_k)\mapsto(q_{k+1},p_{k+1})$, where the discrete Lagrangian is the Type I generating function of the symplectic transformation.

\subsection{Desirable Properties of Variational Integrators}
We summarize the many geometric structure-preserving properties of variational integrators, which account for their popularity in simulations involving mechanical systems.
\begin{description}
\item [\it Symplecticity.] Given a discrete Lagrangian $L_d$, one obtains a discrete fiber derivative, $\mathbb{F}L_d:(q_0,q_1)\mapsto(q_0, -D_1 L_d(q_0,q_1))$. Variational integrators are symplectic, i.e, the pullback under $\mathbb{F}L_d$ of the canonical symplectic form $\Omega$ on the cotangent bundle $T^*Q$, is preserved. Pushing-forward the discrete Euler--Lagrange equations yield a symplectic-partitioned Runge--Kutta method.

\item [\it Momentum Conservation.] Noether's theorem states that if a Lagrangian is invariant under the lifted action of a Lie group, the associated momentum is preserved by the flow. If a discrete Lagrangian is invariant under the diagonal action of a symmetry group, a discrete Noether's theorem holds, and the discrete flow preserves the discrete momentum map. For PDEs with a uniform spatial discretization, a backward error analysis implies approximate spatial momentum conservation \citep{OlWeWu2004}. 

\item [\it Approximate Energy Conservation.] While variational integrators do not exactly preserve energy, backward error analysis \citep{BeGi1994, Ha1994, HaLu1997, Re1999} shows that it preserves a modified Hamiltonian that is close to the original Hamiltonian for exponentially long times. In practice,  the energy error is bounded and does not exhibit a drift. This is the temporal analogue of the approximate momentum conservation result for PDEs, as energy is the momentum map associated with time invariance.

\item [\it Applicable to a Large Class of Problems.] The discrete variational approach is very general, and allows for the construction of geometric structure-preserving numerical integrators for PDEs \citep{LeMaOrWe2003}, nonsmooth collisions \citep{FeMaOrWe2003}, stochastic systems \citep{BROw2009}, nonholonomic systems \citep{CoMa2001}, and constrained systems \citep{LeMaOr2008}. Furthermore, Dirac structures and mechanics allows for interconnections between Lagrangian systems, thereby providing a unified simulation framework for multiphysics systems.

\item [\it Generates a Large Class of Methods.] A variational integrator can be constructed by choosing a finite-dimensional function space, and a numerical quadrature method~\citep{Le2004}. By leveraging techniques from approximation theory, numerical analysis, and finite elements, one can construct variational integrators that are appropriate for problems that evolve on Lie groups \citep{LeLeMc2007a, LeLeMc2007b} and homogeneous spaces \citep{LeLeMc2009b}, or exhibit multiple timescales \citep{StGr2009, LeMaOrWe2003}.
\end{description}

\subsection{Variational Error Analysis and Discrete Noether's Theorem}

The variational integrator approach to constructing symplectic integrators has a few important advantages from the point of view of numerical analysis. In particular, the task of proving properties of the discrete Lagrangian map $F_{L_d}:Q\times Q\rightarrow Q\times Q$ reduces to verifying certain properties of the discrete Lagrangian instead. Here, we summarize the results from Theorems 1.3.3 and 2.3.1 of \cite{MaWe2001} that relate to the order of accuracy and momentum conservation properties of the variational integrator.

\paragraph{Discrete Noether's theorem.} Given a discrete Lagrangian $L_d:Q\times Q\rightarrow \mathbb{R}$ which is invariant under the diagonal action of a Lie group $G$ on $Q\times Q$, then the discrete Lagrangian momentum map, $J_{L_d}:Q\times Q\rightarrow \mathfrak{g}^*$, given by
\[ J_{L_d}(q_k, q_{k+1})\cdot \xi=\langle -D_1 L_d(q_k, q_{k+1}, \xi_Q(q_k)\rangle \]
is invariant under the discrete Lagrangian map, i.e.,  $J_{L_d}\circ F_{L_d}=J_{L_d}$.

\paragraph{Variational error analysis.}
The natural setting for analyzing the order of accuracy of a variational integrator is the variational error analysis framework introduced in \cite{MaWe2001}. In particular, Theorem 2.3.1 of \cite{MaWe2001} states that if a discrete Lagrangian, $L_d:Q\times Q\rightarrow\mathbb{R}$, approximates the exact discrete Lagrangian, $L_d^E:Q\times Q\rightarrow\mathbb{R}$, given in \eqref{exact_Ld_Jacobi} and \eqref{exact_Ld_variational} to order $p$, i.e.,
\[ L_d(q_0, q_1;h)=L_d^E(q_0,q_1;h)+\mathcal{O}(h^{p+1}),\]
then the discrete Hamiltonian map, $\tilde{F}_{L_d}:(q_k,p_k)\mapsto(q_{k+1},p_{k+1})$, viewed as a one-step method, is order $p$ accurate.

\section{Galerkin Variational Integrators}
The variational characterization of the exact discrete Lagrangian~\eqref{exact_Ld_variational},
\[L_d^E(q_0,q_1;h)=\ext_{\substack{q\in C^2([0,h],Q) \\ q(0)=q_0, q(h)=q_1}} \int_0^h L(q(t), \dot q(t)) dt,\]
naturally leads to the Galerkin approach to constructing variational integrators. This involves replacing the infinite-dimensional function space $C^2([0,h],Q)$ with a suitably chosen finite-dimensional subspace, and the integral with a numerical quadrature formula. In particular, it is common to choose an interpolatory approximation space, so as to satisfy the boundary conditions at time $0$ and $h$ in a straightforward manner. 

We now recall the construction of higher-order Galerkin
variational integrators, as originally described in \cite{MaWe2001}.
Given a Lie group $G$, the associated {\bfi state space} is
given by the tangent bundle $TG$. In addition, the dynamics on $G$
is described by a {\bfi Lagrangian}, $L:TG\rightarrow\mathbb{R}$.
Given a time interval $[0,h]$, the {\bfi action map},
$\mathfrak{S}:C^2([0,h],G)\rightarrow\mathbb{R}$, is given by
\[\mathfrak{S}(g)\equiv \int_0^h L(g(t),\dot g(t)) dt.\]

We approximate the action map, by numerical quadrature, to yield $\mathfrak{S}^s:C^2([0,h],G)\rightarrow
\mathbb{R}$,
\[\mathfrak{S}^s(g)\equiv h \sum\nolimits_{i=1}^s b_i L(g(c_i h), \dot g(c_i
h)),\] where $c_i\in[0,1]$, $i=1,\ldots, s$ are the quadrature
points, and $b_i$ are the quadrature weights.

Recall that the discrete Lagrangian should be an approximation of the form
\[L_d(g_0,g_1,h)\approx \ext_{\substack{g\in C^2([0,h],G),\\g(0)=g_0, g(h)=g_1}} \mathfrak{S}(g)\, .\]
If we restrict the extremization procedure to the subspace spanned
by the interpolatory function that is parameterized by $s+1$
internal points, $\varphi:G^{s+1}\rightarrow
C^2([0,h],G)$, we obtain the following discrete
Lagrangian,
\begin{align*}
L_d(g_0,g_1) &= \ext_{\substack{g^\nu\in G;\\g^0=g_0; g^s=g_1}} \mathfrak{S}
(\varphi(g^\nu;\cdot))\\
&= \ext_{\substack{g^\nu\in G;\\ g^0=g_0; g^s=g_1}} h\sum\nolimits_{i=1}^s b_i
L(T\varphi(g^\nu;c_i h)).
\end{align*}

\subsection{Galerkin Lie Group Variational Integrators}

A particularly novel and nontrivial application of this approach is the construction of Lie group variational integrators. This is based on the underlying idea of Lie group integrators~\cite{IsMuNoZa2000}, which is to express the update map of the numerical scheme in terms of the exponential map,
\[ g_1 = g_0 \exp(\xi_{01})\, ,\]
and thereby reduce the problem to finding an appropriate Lie
algebra element $\xi_{01}\in\mathfrak{g}$, such that the update
scheme has the desired order of accuracy. This is a desirable
reduction, as the Lie algebra is a vector space, and as such the
interpolation of elements can be easily defined. In our
construction, the interpolatory method we use on the Lie group
relies on interpolation at the level of the Lie algebra.

Since the momentum conservation properties of variational integrators depend on the discrete Lagrangian being invariant under the diagonal action of the symmetry group, we will introduce a construction that yields a $G$-invariant discrete Lagrangian whenever the continuous Lagrangian is $G$-invariant. This is achieved through the use of $G$-equivariant interpolatory functions, and in particular, natural charts on $G$, which we will now discuss. This then leads to the issue of reduction of higher-order Lie group integrators.

\paragraph{Group-equivariant interpolants.} The interpolatory function is $G$-equivariant if
\[\varphi(g g^\nu;t)=g \varphi( g^\nu;t ).\]
\begin{proposition}\label{gvi:lemma:invariant_Ld}
If the interpolatory function $\varphi(g^\nu;t)$ is
$G$-equivariant, and the Lagrangian, $L:TG\rightarrow\mathbb{R}$, is
$G$-invariant, then the discrete Lagrangian, $L_d:G\times
G\rightarrow \mathbb{R}$, given by
\[
L_d(g_0,g_1)= \ext_{\substack{g^\nu\in G;\\ g^0=g_0; g^s=g_1}} h\sum\nolimits_{i=1}^s
b_i L(T\varphi(g^\nu;c_i h)),
\]
is $G$-invariant.
\end{proposition}
\begin{proof}
\begin{align*}
L_d(g g_0, g g_1) &= \ext_{\substack{\tilde g^\nu\in G;\\ \tilde g^0=g g_0;
\tilde g^s=g g_1}}
h\sum\nolimits_{i=1}^s b_i L(T\varphi(\tilde g^\nu;c_i h)),\\
&= \ext_{\substack{g^\nu\in g^{-1}G;\\ gg^0=gg_0; gg^s=gg_1}}
h\sum\nolimits_{i=1}^s b_i L(T\varphi( g g^\nu;c_i h)),\\
&= \ext_{\substack{g^\nu\in G;\\ g^0=g_0; g^s=g_1}}
h\sum\nolimits_{i=1}^s b_i L(TL_g\cdot T\varphi(g^\nu;c_i h)),\\
&= \ext_{\substack{g^\nu\in G;\\ g^0=g_0; g^s=g_1}}
h\sum\nolimits_{i=1}^s b_i L(T\varphi(g^\nu;c_i h)),\\
&= L_d(g_0,g_1),
\end{align*}
where we used the $G$-equivariance of the interpolatory function
in the third equality, and the $G$-invariance of the Lagrangian in
the forth equality.
\end{proof}

\begin{remark}
While $G$-equivariant interpolatory functions provide a computationally efficient method of constructing $G$-invariant discrete Lagrangians, we can construct a $G$-invariant discrete Lagrangian (when $G$ is compact) by averaging an arbitrary discrete Lagrangian. In particular, given a discrete Lagrangian $L_d:Q\times Q\rightarrow \mathbb{R}$, the averaged discrete Lagrangian, given by
\[ \bar{L}_d(q_0,q_1) = \frac{1}{|G|}\int_{g\in G} L_d(g q_0, g q_1) dg\]
is $G$-equivariant. Therefore, in the case of compact symmetry groups, a $G$-invariant discrete Lagrangian always exists.
\end{remark}

\paragraph{Natural charts.}\label{gvi:subsec:natural_charts}
Following the construction in \cite{MaPeSh1999}, we use the group
exponential map at the identity, $\exp_e:\mathfrak{g}\rightarrow
G$, to construct a $G$-equivariant interpolatory function, and a
higher-order discrete Lagrangian. As shown in
Lemma~\ref{gvi:lemma:invariant_Ld}, this construction yields a
$G$-invariant discrete Lagrangian if the Lagrangian itself is
$G$-invariant.

In a finite-dimensional Lie group $G$, $\exp_e$ is a local
diffeomorphism, and thus there is an open neighborhood $U\subset
G$ of $e$ such that $\exp_e^{-1}:U\rightarrow
\mathfrak{u}\subset\mathfrak{g}$. When the group acts on the left,
we obtain a chart $\psi_g:L_g U \rightarrow \mathfrak{u}$ at $g\in
G$ by
\[\psi_g=\exp_e^{-1}\circ L_{g^{-1}}.\]
We would like to construct an interpolatory function that is described by a set of control points $\{g^\nu\}_{\nu=0}^s$ in the group $G$ at control times $0=d_0<d_1<d_2<\ldots<d_{s-1}<d_s=1$. Our natural chart based at $g^0$ induces a set of control points $\xi^\nu=\psi_{g^0}^{-1}(g^\nu)$ in the Lie algebra $\mathfrak{g}$ at the same control times. Let $\tilde{l}_{\nu,s}(t)$ denote the Lagrange polynomials associated with the control times $d_\nu$, which yields an interpolating polynomial at the level of the Lie algebra,
\[ \xi_d(\xi^\nu;\tau h)=\sum\nolimits_{\nu=0}^s \xi^\nu \tilde{l}_{\nu,s}(\tau).\]
Applying $\psi_{g^0}^{-1}$ yields an interpolating curve in $G$ of the form,
\[\varphi(g^\nu;\tau h)=\psi_{g^0}^{-1}
\Big(\sum\nolimits_{\nu=0}^s
\psi_{g^0}(g^\nu)\tilde{l}_{\nu,s}(\tau)\Big),\]
where $\varphi(d_\nu h)=g^\nu$ for $\nu=0,\ldots, s$. Furthermore, this interpolant is $G$-equivariant, as shown in the following Lemma.
\begin{proposition}
The interpolatory function given by
\[\varphi(g^\nu;\tau h)=\psi_{g^0}^{-1}
\Big(\sum\nolimits_{\nu=0}^s
\psi_{g^0}(g^\nu)\tilde{l}_{\nu,s}(\tau)\Big),\] is
$G$-equivariant.
\end{proposition}
\begin{proof}
\begin{align*}
\varphi(gg^\nu;\tau h) &= \psi_{(gg^0)}^{-1}
\Big(\sum\nolimits_{\nu=0}^s
\psi_{gg^0}(gg^\nu)\tilde{l}_{\nu,s}(\tau)\Big)\\
&= L_{gg^0} \exp_e \Big(\sum\nolimits_{\nu=0}^s
\exp_e^{-1}((gg^0)^{-1}(gg^\nu))\tilde{l}_{\nu,s}(\tau)\Big)\\
&= L_g L_{g^0} \exp_e \Big(\sum\nolimits_{\nu=0}^s
\exp_e^{-1}((g^0)^{-1}g^{-1}gg^\nu)\tilde{l}_{\nu,s}(\tau)\Big)\\
&= L_g \psi_{g^0}^{-1} \Big(\sum\nolimits_{\nu=0}^s
\exp_e^{-1}\circ L_{(g^0)^{-1}}(g^\nu)\tilde{l}_{\nu,s}(\tau)\Big)\\
&= L_g \psi_{g^0}^{-1} \Big(\sum\nolimits_{\nu=0}^s
\psi_{g^0}(g^\nu)\tilde{l}_{\nu,s}(\tau)\Big)\\
&= L_g \varphi(g^\nu;\tau h).
\end{align*}
\end{proof}
\begin{remark}
In the proof that $\varphi$ is $G$-equivariant, it was important
that the base point for the chart should transform in the same way
as the internal points $g^\nu$. As such, the interpolatory
function will be $G$-equivariant for a chart that it based at any
one of the internal points $g^\nu$ that parameterize the function,
but will not be $G$-equivariant if the chart is based at a fixed
$g\in G$. Without loss of generality, we will consider the case
when the chart is based at the first point $g_0$.
\end{remark}

We will now consider a discrete Lagrangian based on the use of interpolation on a natural chart, which is given by
\[ L_d(g_0, g_1) = \ext_{\substack{g^\nu\in G; \\g^0=g_0; g^s = g_0^{-1}g_1}} h \sum\nolimits_{i=1}^s b_i L(T\varphi (\{g^\nu\}_{\nu=0}^s;c_i h))\, .\]
To further simplify the expression, we will express the extremal in terms of the Lie algebra elements $\xi^\nu$ associated with the $\nu$-th control point. This relation is given by
\[\xi^\nu=\psi_{g_0}(g^\nu)\, ,\]
and the interpolated curve in the algebra is given by
\[\xi(\xi^\nu;\tau h)=\sum\nolimits_{\kappa=0}^s \xi^\kappa \tilde l_{\kappa,s}(\tau),\]
which is related to the curve in the group,
\[g(g^\nu;\tau h)=g_0 \exp(\xi(\psi_{g_0}(g^\nu);\tau h)).\]
The velocity $\dot \xi=g^{-1} \dot g$ is given by
\[\dot\xi(\tau h)=g^{-1}\dot g(\tau h)=\frac{1}{h}\sum\nolimits_{\kappa=0}^s \xi^\kappa \dot{\tilde
l}_{\kappa,s}(\tau).\]
Using the standard formula for the
derivative of the exponential,
\[T_\xi \exp= T_e L_{\exp(\xi)}\cdot \dexp_{\ad_{\xi}},\]
where
\[\dexp_w=\sum\nolimits_{n=0}^\infty \frac{w^n}{(n+1)!},\]
we obtain the following expression for discrete Lagrangian,
\begin{multline*}
L_d(g_0, g_1) = \ext_{\substack{\xi^\nu\in\mathfrak{g};\\\xi^0=0;\xi^s=\psi_{g_0}(g_1)}}
h \sum\nolimits_{i=1}^s b_i L\Big(L_{g_0} \exp(\xi(c_i h)),\\
T_{\exp(\xi(c_i h))}L_{g_0}\cdot T_e
L_{\exp(\xi(c_i h))}\cdot \dexp_{\ad_{\xi(c_i h)}} (\dot \xi(c_i
h))\Big)\, .
\end{multline*}
More explicitly, we can compute the conditions on the Lie algebra elements for the expression above to be extremal. This implies that
\begin{multline*}
L_d(g_0, g_1) = h \sum\nolimits_{i=1}^s b_i L\Big(L_{g_0} \exp(\xi(c_i h)),\\
T_{\exp(\xi(c_i h))}L_{g_0}\cdot T_e L_{\exp(\xi(c_i h))}\cdot \dexp_{\ad_{\xi(c_i h)}} (\dot \xi(c_i h))\Big)
\end{multline*}
with $\xi^0=0$, $\xi^s=\psi_{g_0}(g_1)$, and the other Lie algebra elements implicitly defined by
\begin{multline*}
0 = h\sum\nolimits_{i=1}^s b_i \left [\frac{\partial L}{\partial g} (c_i h) T_{\exp(\xi(c_i h))} L_{g_0}\cdot T_e L_{\exp(\xi(c_i h))} \cdot \dexp_{\ad_{\xi(c_i h)}} \tilde{l}_{\nu,s}(c_i) \right.\\
\qquad+ \left . \frac{1}{h}\frac{\partial L}{\partial
\dot g}(c_i h) T^2_{\exp(\xi(c_i h))} L_{\exp(\xi(c_i h))}\cdot
T^2_e L_{\exp(\xi(c_i h))}\cdot \ddexp_{\ad_{\xi(c_i
h)}}\dot{\tilde{l}}_{\nu,s}(c_i)\right ] ,
\end{multline*}
for $\nu=1,\ldots,s-1$, and where
\[ \ddexp_w = \sum\nolimits_{n=0}^\infty \frac{w^n}{(n+2)!}\, . \]
This expression for the higher-order discrete Lagrangian, together with the discrete Euler--Lagrange equation,
\[ D_2 L_d(g_0,g_1) + D_1 L_d(g_1, g_2)=0\, ,\]
yields a {\bfi higher-order Lie group variational integrator}.

\subsection{Higher-Order Discrete Euler--Poincar\'e Equations}\index{variational integrator!Euler--Poincar\'e}\index{reduction!Euler--Poincar\'e}
In this section, we will apply discrete Euler--Poincar\'e reduction (see, for example,~\cite{MaPeSh1999}) to the Lie group variational integrator we derived previously, to construct a higher-order generalization of discrete Euler--Poincar\'e reduction.

\paragraph{Reduced discrete Lagrangian.}
We first proceed by computing an expression for the reduced
discrete Lagrangian in the case when the Lagrangian is
$G$-invariant. Recall that our discrete Lagrangian uses
$G$-equivariant interpolation, which, when combined with the
$G$-invariance of the Lagrangian, implies that the discrete
Lagrangian is $G$-invariant as well. We compute the reduced
discrete Lagrangian,
\begin{align*}
l_d(g_0^{-1}g_1)
&\equiv L_d(g_0,g_1)\\
&= L_d(e, g_0^{-1}g_1)\\
&= \ext_{\substack{\xi^\nu\in\mathfrak{g};\\\xi^0=0;\xi^s=\log(g_0^{-1}g_1)}}
h \sum\nolimits_{i=1}^s b_i L\Big(L_{e} \exp(\xi(c_i h)),\\
&\hspace*{1in} T_{\exp(\xi(c_i h))}L_{e}\cdot T_e L_{\exp(\xi(c_i h))}\cdot \dexp_{\ad_{\xi(c_i h)}} (\dot \xi(c_i h))\Big)\\
&= \ext_{\substack{\xi^\nu\in\mathfrak{g};\\ \xi^0=0;\xi^s=\log(g_0^{-1}g_1)}}
h\sum\nolimits_{i=1}^s b_i L\Big(\exp(\xi(c_i h)),\\
&\hspace*{1in} T_e L_{\exp(\xi(c_i
h))}\cdot \dexp_{\ad_{\xi(c_i h)}}(\dot\xi(c_i h))\Big)\, .
\end{align*}
Setting $\xi^0=0$, and $\xi^s=\log(g_0^{-1}g_1)$, we can solve the
stationarity conditions for the other Lie algebra elements
$\{\xi^\nu\}_{\nu=1}^{s-1}$ using the following implicit system of
equations,
\begin{multline*}
0 = h\sum\nolimits_{i=1}^s b_i \left [\frac{\partial L}{\partial g} (c_i h) T_e L_{\exp(\xi(c_i h))} \cdot \dexp_{\ad_{\xi(c_i h)}} \tilde{l}_{\nu,s}(c_i) \right.\\
+ \left . \frac{1}{h}\frac{\partial L}{\partial
\dot g}(c_i h) T^2_e L_{\exp(\xi(c_i h))}\cdot
\ddexp_{\ad_{\xi(c_i h)}}\dot{\tilde{l}}_{\nu,s}(c_i)\right ]
\end{multline*}
where $\nu=1,\ldots, s-1$.

This expression for the reduced discrete Lagrangian is not fully
satisfactory however, since it involves the Lagrangian, as opposed
to the reduced Lagrangian. If we revisit the expression for the
reduced discrete Lagrangian,
\begin{multline*}
l_d(g_0^{-1}g_1)=\\
\ext_{\substack{\xi^\nu\in\mathfrak{g};\\\xi^0=0;\xi^s=\log(g_0^{-1}g_1)}}
h\sum\nolimits_{i=1}^s b_i L\Big(\exp(\xi(c_i h)), T_e L_{\exp(\xi(c_i
h))}\cdot \dexp_{\ad_{\xi(c_i h)}}(\dot\xi(c_i h))\Big)\, ,
\end{multline*}
we
find that by $G$-invariance of the Lagrangian, each of the terms
in the summation,
\[ L\Big(\exp(\xi(c_i h)), T_e L_{\exp(\xi(c_i
h))}\cdot \dexp_{\ad_{\xi(c_i h)}}(\dot\xi(c_i h))\Big)\, ,\] can
be replaced by
\[ l\Big(\dexp_{\ad_{\xi(c_i h)}}(\dot\xi(c_i h))\Big)\, ,\]
where $l:\mathfrak{g}\rightarrow \mathbb{R}$ is the {\bfi reduced
Lagrangian} given by
\[ l(\eta)= L(L_{g^{-1}}g,TL_{g^{-1}}\dot g) = L(e,\eta), \]
where $\eta=TL_{g^{-1}}\dot g\in\mathfrak{g}$.

From this observation, we have an expression for the reduced
discrete Lagrangian in terms of the reduced Lagrangian,
\[l_d(g_0^{-1}g_1) =
\ext_{\substack{\xi^\nu\in\mathfrak{g};\\\xi^0=0;\xi^s=\log(g_0^{-1}g_1)}}
h\sum\nolimits_{i=1}^s b_i l\Big(\dexp_{\ad_{\xi(c_i h)}}(\dot\xi(c_i
h))\Big)\, .\] As before, we set $\xi^0=0$, and
$\xi^s=\log(g_0^{-1}g_1)$, and solve the stationarity conditions
for the other Lie algebra elements $\{\xi^\nu\}_{\nu=1}^{s-1}$
using the following implicit system of equations,
\begin{align*}
0 &= h\sum\nolimits_{i=1}^s b_i \left [\frac{\partial l}{\partial \eta}
(c_i h) \ddexp_{\ad_{\xi(c_i
h)}}\dot{\tilde{l}}_{\nu,s}(c_i)\right ]\, ,
\end{align*}
where $\nu=1,\ldots, s-1$.

\paragraph{Discrete Euler--Poincar\'e equations.}
As shown above, we have constructed a higher-order reduced
discrete Lagrangian that depends on
\[ f_{kk+1}\equiv g_k g_{k+1}^{-1}.\]
We will now recall the derivation of the discrete Euler--Poincar\'e equations, introduced in \cite{MaPeSh1999}. The variations in $f_{kk+1}$ induced by variations in $g_k$, $g_{k+1}$ are computed as follows,
\begin{align*}
\delta f_{kk+1} &= -g_k^{-1} \delta g_k g_k{-1} g_{k+1} + g_k^{-1}\delta g_{k+1}\\
&= TR_{f_{kk+1}}(- g_k^{-1}\delta g_k + \Ad_{f_{kk+1}} g_{k+1}\delta g_{k+1})\, .
\end{align*}
Then, the variation in the discrete action sum is given by
\begin{align*}
\delta \mathbb{S}
&= \sum\nolimits_{k=0}^{N-1} l'_d (f_{kk+1}) \delta f_{kk+1}\\
&= \sum\nolimits_{k=0}^{N-1} l'_d (f_{kk+1}) TR_{f_{kk+1}}(- g_k^{-1}\delta g_k + \Ad_{f_{kk+1}} g_{k+1}\delta g_{k+1})\\
&= \sum\nolimits_{k=1}^{N-1} \left [ l'_d(f_{k-1k})TR_{f_{k-1k}}\Ad_{f_{k-1k}} - l'_d(f_{kk+1})TR_{f_{kk+1}}\right ] \vartheta_k \, ,
\end{align*}
with variations of the form $\vartheta_k = g_k^{-1} \delta g_k$. In computing the variation of the discrete action sum, we have collected terms involving the same variations, and used the fact that $\vartheta_0=\vartheta_N=0$. This yields the {\bfi discrete Euler--Poincar\'e equation},
\begin{align*}
l'_d(f_{k-1k})TR_{f_{k-1k}}\Ad_{f_{k-1k}} -
l'_d(f_{kk+1})TR_{f_{kk+1}}&=0, & k&=1,\ldots, N-1.
\end{align*}
For ease of reference, we will recall the expressions from the
previous discussion that define the {\bfi higher-order reduced
discrete Lagrangian},
\[
l_d(f_{kk+1}) =h\sum\nolimits_{i=1}^s b_i l\Big(\dexp_{\ad_{\xi(c_i
h)}}(\dot\xi(c_i h))\Big)\, ,
\]
where
\[ \xi(\xi^\nu;\tau h) = \sum\nolimits_{\kappa=0}^s \xi^\kappa \tilde{l}_{\kappa,s}(\tau)\, ,\]
and
\begin{align*}
\xi^0 &= 0\, ,\\
\xi^s &= \log(f_{kk+1})\, ,
\end{align*}
and the remaining Lie algebra elements $\{\xi^\nu\}_{\nu=1}^{s-1}$, are defined implicitly by
\begin{align*}
0 &= h\sum\nolimits_{i=1}^s b_i \left [\frac{\partial l}{\partial \eta}
(c_i h) \ddexp_{\ad_{\xi(c_i
h)}}\dot{\tilde{l}}_{\nu,s}(c_i)\right ],
\end{align*}
for $\nu=1,\ldots,s-1$, and where
\[
\ddexp_w = \sum\nolimits_{n=0}^\infty \frac{w^n}{(n+2)!}\, .
\]
When the discrete Euler--Poincar\'e equation is used in
conjunction with the higher-order reduced discrete Lagrangian, we
obtain the {\bfi higher-order Euler--Poincar\'e equations}.

\subsection{Example: Lie Group Velocity Verlet}\label{comp}
We will now construct a Lie group analogue of the velocity Verlet method for the free rigid body. The velocity Verlet method can be derived from the context of discrete mechanics by considering the following discrete Lagrangian,
\[ L_d(q_k, q_{k+1})=\frac{h}{2}\left[L\left(q_k,\frac{q_{k+1}-q_k}{h}\right)+L\left(q_{k+1},\frac{q_{k+1}-q_k}{h}\right)\right],\]
which corresponds to using a piecewise linear interpolant, and the trapezoidal rule to approximate the integral.

In the case of the free rigid body, the Lagrangian is given by,
\[ L(R,\dot{R})=\frac{1}{2}\Omega J\Omega^T=\frac{1}{2}\tr{S(\Omega) J_d S(\Omega)^T}
.\]
Here, $J_d$ is a modified moment of inertia that is related to the usual moment of inertia by the relations, $J_d=\frac{1}{2}(\tr{J}I_{3\times 3}-2J)$, and $J=\tr{J_d}I_{3\times 3}-J_d$. From the kinematic relation $S(\Omega)=R^T\dot{R}$, we have that,
\[ S(\Omega_k) = R_k^T \dot R_k \approx R_k\frac{R_{k+1}-R_k}{h}=\frac{1}{h}(F_k-I_{3\times 3}),\]
where $F_k=R_k^T R_{k+1}$. Then, the discrete Lagrangian for the velocity Verlet method applied to the free rigid body is given by,
\begin{align*}
L_d(R_k, R_{k+1})
&=2\cdot\frac{h}{2}\frac{1}{2}\frac{1}{h^2}\tr{(F_k-I_{3\times 3})^T J_d(F_k-I_{3\times 3})}\\
&=\frac{1}{2h}\tr{(F_k-I_{3\times 3})(F_k-I_{3\times 3})^T J_d}\\
&=\frac{1}{h}\tr{(I_{3\times 3}-F_k)J_d},
\end{align*}
where in the second to last equality, we used the fact that $\tr{AB}=\tr{BA}$, and in the last equality, we used the fact that $F_k$ is an orthogonal matrix, $J_d$ is symmetric, and $\tr{AB}=\tr{B^T A^T}$.

Recall that $\frac{\partial R^{T}}{\partial R}\cdot \delta R=-R^{T}(\delta R) R^{T}$. Furthermore, the variation of $R_k$ is given by,
\[\delta R_k=R_k \eta_k,\]
where $\eta_k\in\mathfrak{so}(3)$ is a variation represented by a skew-symmetric matrix and vanishes at $k=0$ and $k=N$. We may now compute the constrained variation of $F_k=R_k^T R_k$, which yields,
\begin{align*}
\delta F_k&=\delta R_k^T R_{k+1}+R_k^T \delta R_{k+1}
=\eta_k R_k^T R_{k+1}+R_k^T R_{k+1}\eta_{k+1}
=-\eta_k F_k + F_k \eta_{k+1}.
\end{align*}
Define the discrete action sum to be
\[ \mathfrak{S}_d=\sum\nolimits_{k=0}^{N-1}L_d(R_k, F_k).\]
Taking constrained variations of $F_k$ yields,
\[\delta \mathfrak{S}_d=\sum\nolimits_{k=0}^{N-1}\frac{1}{h}\left\{\tr{-\eta_{k+1}J_d F_k}+\tr{\eta_k F_k J_d}\right\}.\]
Using the fact that the variations $\eta_k$ vanish at the endpoints, we may reindex the sum to obtain,
\[\delta \mathfrak{S}_d=\sum\nolimits_{k=1}^{N-1}\frac{1}{h}\tr{\eta_k(F_k J_d - J_d F_{k-1})}.\]
The discrete Hamilton's principle states that the variation of the discrete action sum should be zero for all variations that vanish at the endpoints. Since $\eta_k$ is an arbitrary skew-symmetric matrix, for the discrete action sum to be zero, it is necessary for $(F_k J_d - J_d F_{k-1})$ to be symmetric, which is to say that,
\[F_{k+1}J_d-J_d F_{k+1}^T - J_d F_k +F_k^T J_d=0.\]
This implicit equation for $F_{k+1}$ in terms of $F_k$, together with the reconstruction equation $R_{k+1}=R_k F_k$, yields the Lie group analogue of the velocity Verlet method.

In practice, in time marching the numerical solution, we need to solve the above equation for $F_{k+1}\in SO(3)$ given $F_k$. This equation is linear in $F_{k+1}$, but it is implicit due to the nonlinear constraint
$F_{k+1}^TF_{k+1}=I_{3\times 3}$. Since $J_d F_k -F_k^T J_d$ is a skew-symmetric matrix, it may be represented as $S(g)$, where $g\in\mathbb{R}^3$, which reduces the equation to the form,
\begin{align}
FJ_d - J_d F^T = S(g).\label{eqn:findf}
\end{align}

We now introduce two iterative approaches to solve \refeqn{findf}
numerically.

\paragraph{Exponential map.}
An element of a Lie group can be expressed as the exponential of an element of its Lie
algebra, so $F\in\SO$ can be expressed as an exponential of $S(f)\in\so$ for some vector
$f\in\mathbb{R}^3$. The exponential can be written in closed form, using Rodrigues' formula,
\begin{align}
F &= \exp{S(f)} = I_{3\times3} + \frac{\sin\norm{f}}{\norm{f}} S(f) +
\frac{1-\cos\norm{f}}{\norm{f}^2}S(f)^2.\label{eqn:rodc}
\end{align}
Substituting \refeqn{rodc} into \refeqn{findf}, we obtain
\begin{align*}
S(g) & = \frac{\sin\norm{f}}{\norm{f}} S(Jf) + \frac{1-\cos\norm{f}}{\norm{f}^2}
S(f\times Jf).
\end{align*}
Thus, \refeqn{findf} is converted into the equivalent vector equation $g=G(f)$, where $G
: \mathbb{R}^3 \mapsto \mathbb{R}^3$ is given by
\begin{align*}
G(f) & = \frac{\sin\norm{f}}{\norm{f}}\, J f + \frac{1-\cos\norm{f}}{\norm{f}^2}\, f
\times J f.
\end{align*}
We use the Newton method to solve $g=G(f)$, which gives the iteration
\begin{align}
f_{i+1} = f_i + \nabla G(f_i)^{-1} \parenth{g- G(f_i)}.\label{eqn:newton}
\end{align}
We iterate until $\norm{g- G(f_i) } < \epsilon$ for a small tolerance $\epsilon > 0$. The
Jacobian $\nabla G(f)$ in \refeqn{newton} can be expressed as
\begin{align*}
\nabla G(f) & = \frac{\cos\norm{f}\norm{f}-\sin\norm{f}}{\norm{f}^3}Jff^T
+ \frac{\sin\norm{f}}{\norm{f}} J\\
& \quad + \frac{\sin\norm{f}\norm{f}-2(1-\cos\norm{f})}{\norm{f}^4}\parenth{f
\times Jf}f^T\\
& \quad +\frac{1-\cos\norm{f}}{\norm{f}^2} \braces{-S(Jf)+S(f)J}.
\end{align*}

\paragraph{Cayley transformation.} Similarly, given $f_c\in\mathbb{R}^3$, the Cayley transformation is a local
diffeomorphism that maps $S(f_c)\in\so$ to $F\in\SO$, where
\begin{align}
    F=\mathrm{cay}\, S(f_c) = (I_{3\times 3}+S(f_c))(I_{3\times
    3}-S(f_c))^{-1}.\label{eqn:FCay}
\end{align}
Substituting \refeqn{FCay} into \refeqn{findf}, we obtain a vector equation $G_c(f_c)=0$
equivalent to \refeqn{findf}
\begin{align}
    G_c(f_c)=g+g\times f_c +(g^Tf_c)f_c-2Jf_c=0,\label{eqn:Gc}
\end{align}
and its Jacobian $\nabla G_c(f_c)$ is written as
\begin{align*}
    \nabla G_c (f_c) = S(g)+(g^Tf_c)I_{3\times 3} +f_c g^T -2J.
\end{align*}
Then, \refeqn{Gc} is solved by using Newton's iteration \refeqn{newton}, and the rotation
matrix is obtained by the Cayley transformation.

For both methods, numerical experiments show that 2 or 3 iterations are sufficient to
achieve a tolerance of $\epsilon=10^{-15}$. Numerical iteration with the Cayley
transformation is a faster by a factor of 4-5 due to the simpler expressions in the iteration. It should be noted that since $F=\exp S(f)$ or $F=\mathrm{cay}\,S(f_c)$, it is automatically a rotation matrix,
even when the equation $g=G(f)$ is not satisfied to machine precision. This computational approach is distinct from directly solving the implicit equation \refeqn{findf} with 9 variables and 6 constraints.

\subsection{Numerical comparisons}
The Lie Group Velocity Verlet method is a second-order symplectic Lie group method, and it is natural to compare it to other second-order accurate methods which fail to preserve either the symplectic or Lie group structure:
\begin{enumerate}[i.]
\item Explicit Midpoint Rule (RK): Preserves neither symplectic nor Lie group properties.
\item Implicit Midpoint Rule (SRK): Symplectic but does not preserve Lie group properties.
\item Crouch-Grossman (LGM): Lie group method but not symplectic.
\end{enumerate}
We consider the energy conservation properties of these integrators, and the extent to which they stay on the rotation group, as a function of step-size. Furthermore, we will explore how the computational cost scales as the step-size is varied.

\begin{figure}[b]
\begin{center}
\subfigure[Computed total energy for 30 seconds]{\includegraphics[height=1.7in]{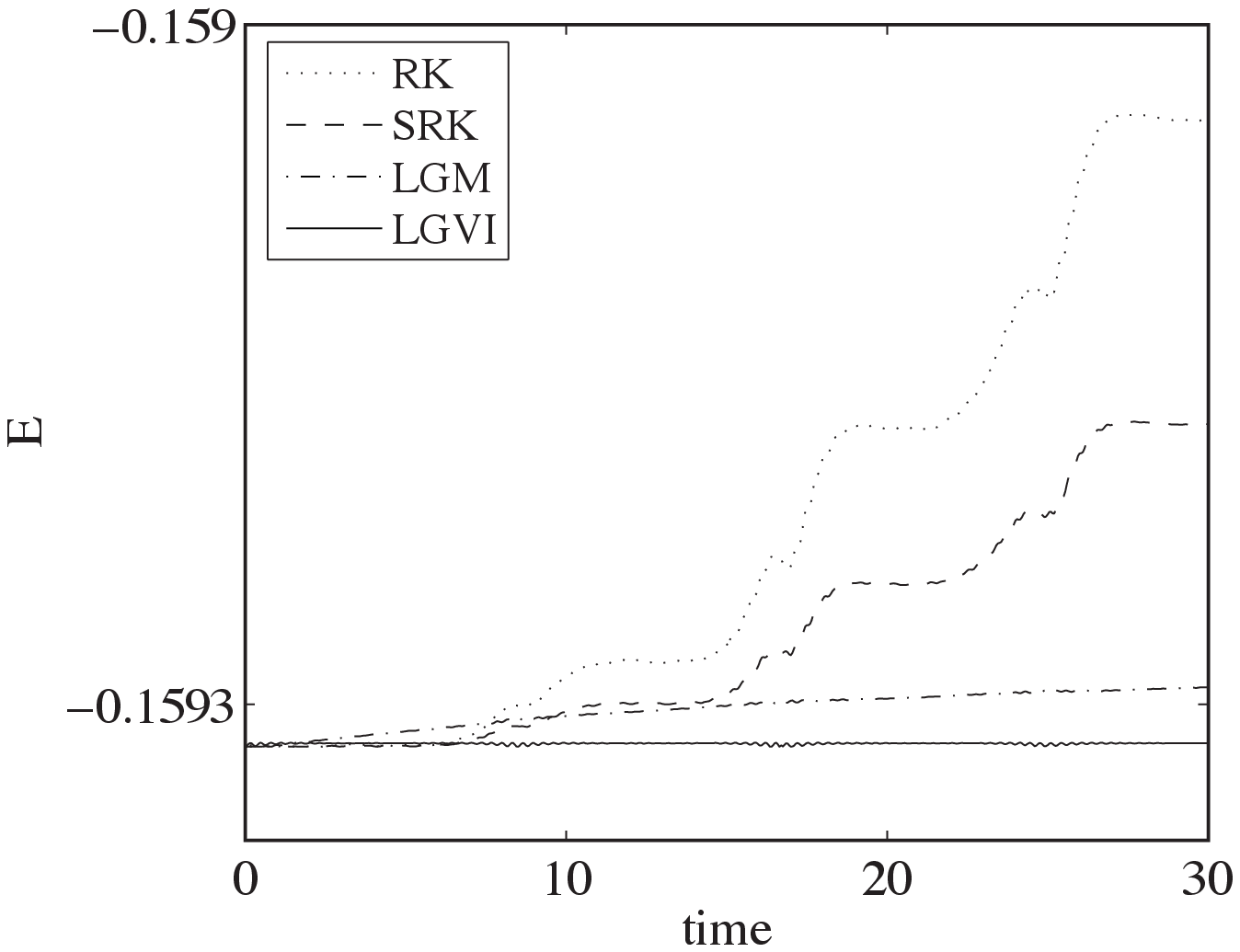}}\qquad
\subfigure[Mean orthogonality error $\|I-R^T R\|$ vs. step size]{\includegraphics[height=1.7in]{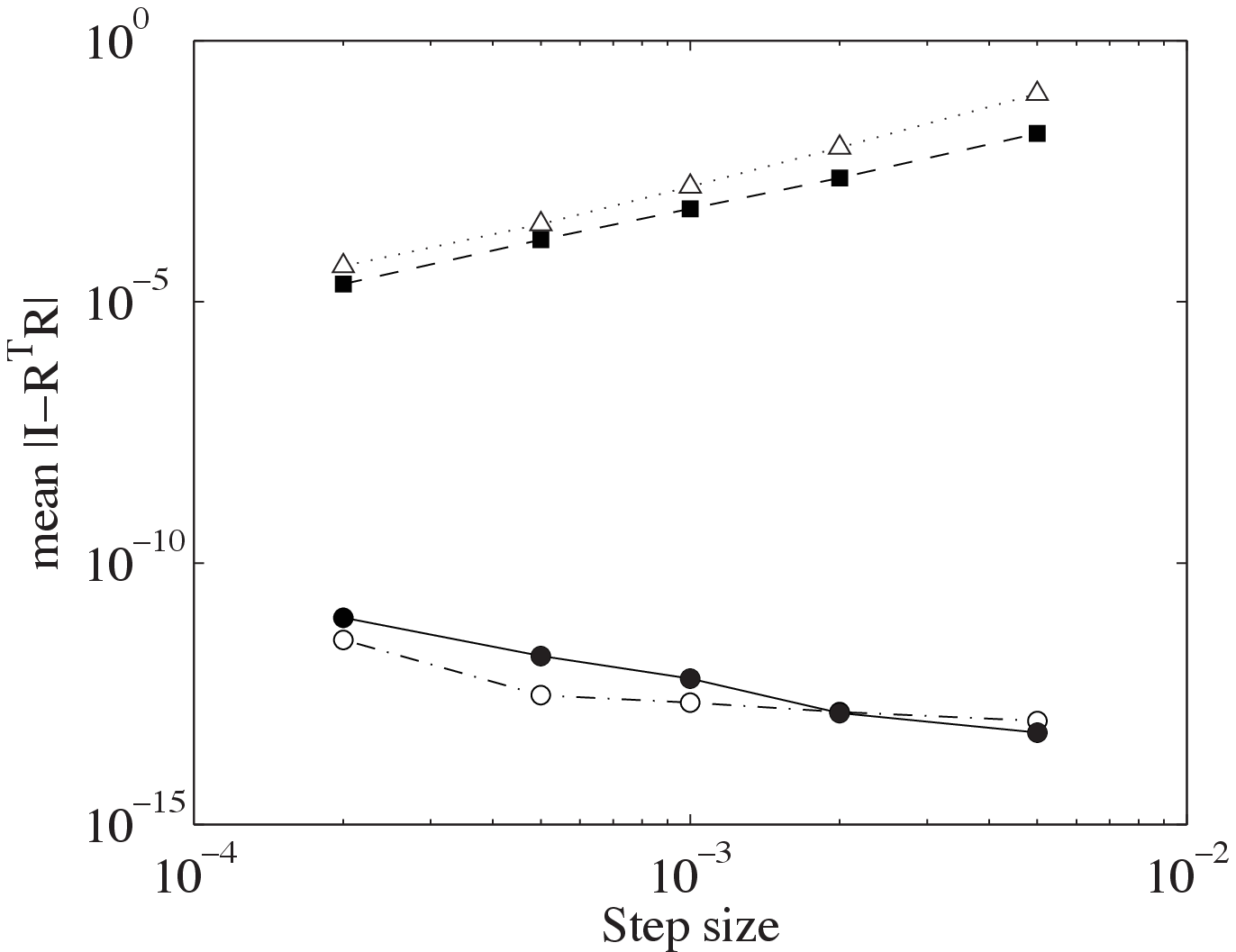}}\\
\subfigure[Mean total energy error $|E-E_0|$ vs. step size]{\includegraphics[height=1.7in]{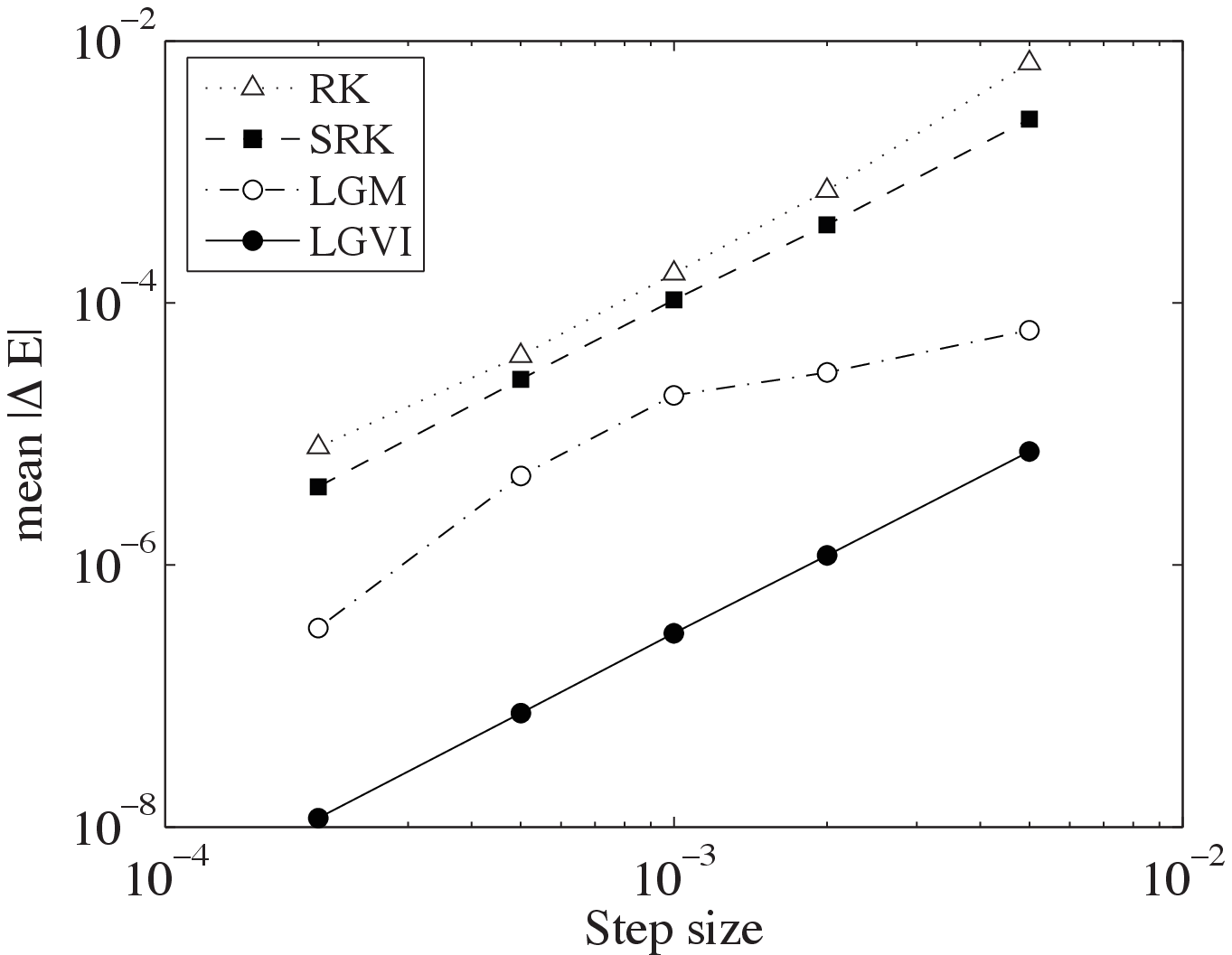}}\qquad
\subfigure[CPU time vs. step size]{\includegraphics[height=1.7in]{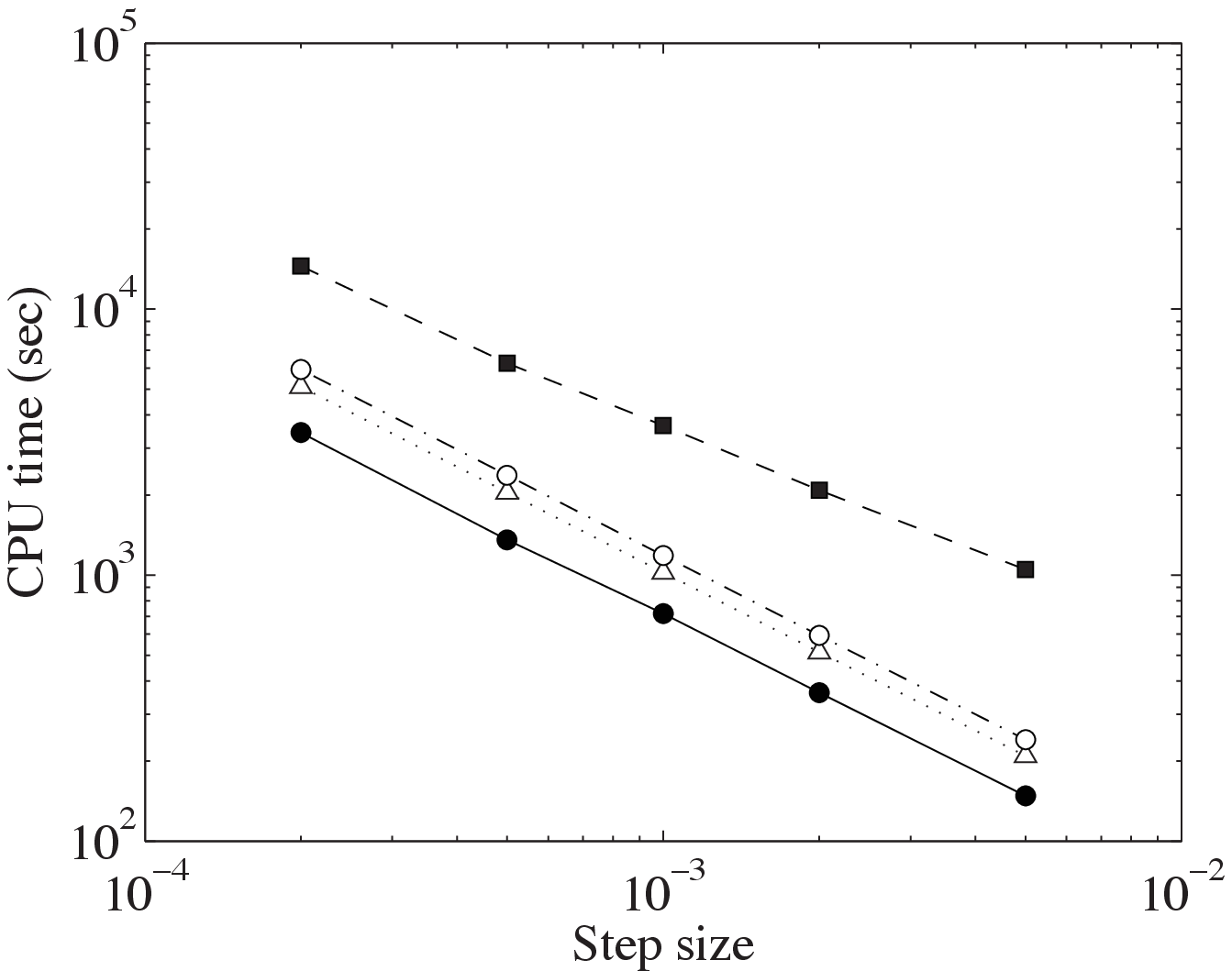}}
\end{center}
\vspace*{-2ex}
\caption{Comparison of Lie Group Velocity Verlet (LGVI) with Explicit Midpoint (RK), Implicit Midpoint (SRK), and Crouch-Grossman (LGM).}
 \label{LGVI_comparison}
\end{figure}

\paragraph{Discussion of numerical results.} Notice that while one typically expects a symplectic method to exhibit good energy behavior, we find that the implicit midpoint method (SRK) does more poorly than the two methods that preserve the Lie group structure---the Crouch-Grossman (LGM) and the Lie Group Velocity Verlet (LGVI) methods. We find that preserving the Lie group structure and the symplectic structure simultaneously yields the best results in terms of energy preservation.

Not surprisingly, the Lie group methods perform the best in terms of the orthogonality error. The increase in orthogonality error as the step-size decreases is associated to the accumulation of round-off error, since smaller step-sizes result in more matrix multiplications. This phenomena can be mitigated by a more careful implementation that utilizes compensated summation~\cite{Ka1965}.

The Lie Group Velocity Verlet method also exhibits the best computational efficiency, since it only requires one force evaluation per time-step, an advantage that it inherits from the vector space version of the Verlet method for separable Hamiltonian systems.

\section{Variational Integrators from Arbitrary One-Step Methods}
This section discusses how one can, given a $p$-th order accurate one-step method for an initial-value problem, and a $q$-th order accurate numerical quadrature formula, construct a variational integrator with  order of accuracy $\min(p,q)$.

A related effort to develop variational integrators from arbitrary one-step methods was introduced in \cite{PaSpZhCu2009}, but this was developed using a nonstandard formulation of discrete variational mechanics on phase space and curve segments \cite{CuPa2009}, which is quite an abstract and involved construction. In contrast, the approach proposed in this section is far more transparent, relies on the well-understood theory of shooting methods for boundary-value problems, and is developed in the standard setting of discrete Lagrangian mechanics~\cite{MaWe2001}. 

\paragraph{Outline of Approach.} Notice that the characterization of the exact discrete Lagrangian associated to Jacobi's solution is expressed in terms of the action integral evaluated on a solution of a two-point boundary-value problem. A standard method of solving a boundary-value problem is to reduce it to the solution of an initial-value problem using the method of shooting. 

We obtain a computable approximation to the exact discrete Lagrangian~\eqref{exact_Ld_Jacobi} in two stages: (i) apply a numerical quadrature formula to the action integral, evaluated along the exact solution of the Euler--Lagrange boundary-value problem; (ii) replace the exact solution of the Euler--Lagrange boundary-value problem by a converged shooting solution associated with a given one-step method.

\paragraph{Shooting-based discrete Lagrangian.}
Given a one-step method $\Psi_h:TQ\rightarrow TQ$, and a numerical quadrature formula $\int_0^h f(x) dx \approx h \sum_{i=0}^n b_i f(x(c_i h))$, with quadrature weights $b_i$ and quadrature nodes $0=c_0<c_1<\ldots<c_{n-1}<c_n=1$, we construct the {\bfi shooting-based discrete Lagrangian},
\begin{equation}
 L_d(q_0, q_1;h)=h\sum\nolimits_{i=0}^n b_i L(q^i, v^i),\label{shooting_Ld_A}
\end{equation}
where
\begin{equation}
(q^{i+1},v^{i+1})=\Psi_{(c_{i+1}-c_i)h}(q^i, v^i),\qquad q^0=q_0, \qquad q^n=q_1.\label{shooting_Ld_B}
\end{equation}
Note that while we formally require that the endpoints are included as quadrature points, i.e., $c_0=0$, and $c_n=1$, the associated weights $b_0$, $b_n$ can be zero, so this is does not constrain the type of quadrature formula we can consider.

\paragraph{Order of discrete Lagrangian.}
The order analysis of the shooting-based discrete Lagrangian depends critically on the global approximation properties of the shooting solution of two-point boundary-value problems.

In general, the Euler--Lagrange equation is a second-order nonlinear differential equation, and it is a standard result in the numerical analysis of the shooting method for nonlinear problems (see, for example, Theorem 2.2.2 of \citep{Keller1968}) that the approximation error in the solution of a boundary-value problem is bounded by the sum of two terms, the first of which is $\mathcal{O}(h^p)$, associated with the global error of the one-step method applied to the initial-value problem, and the second of which is related to the rate of convergence of the nonlinear solver used to determine the appropriate initial condition for the initial-value problem that would lead to the correct terminal boundary condition of the boundary-value problem. In essence, if the solution of the boundary-value problem is isolated and sufficiently regular, then a properly converged shooting method yields a solution of the two-point boundary-value problem that has error $\mathcal{O}(h^p)$, if the underlying one-step method has local truncation error $\mathcal{O}(h^{p+1})$. 

This result allows us to obtain the following theorem on the order of accuracy of the shooting-based discrete Lagrangian, which in turn allows us to establish the order of accuracy of the associated variational integrator.

\begin{theorem}
Given a $p$-th order accurate one-step method $\Psi,$ a $q$-th order accurate quadrature formula, and a Lagrangian $L$ that is Lipschitz continuous in both variables, the associated shooting-based discrete Lagrangian \eqref{shooting_Ld_A}-\eqref{shooting_Ld_B} has order of accuracy $\min(p,q)$.
\end{theorem}
\begin{proof}
By Theorem 2.2.2 of \citep{Keller1968}, a fully converged shooting solution $(\tilde{q}_{01},\tilde{v}_{01})$, associated with a one-step method $\Psi$ of order $p$, approximates the exact solution  $(q_{01},v_{01})$ of the Euler--Lagrange boundary-value problem with the following global error,
\begin{align*}
q_{01}(c_i h)&=\tilde{q}_{01}(c_i h)+\mathcal{O}(h^p),\\
v_{01}(c_i h)&=\tilde{v}_{01}(c_i h)+\mathcal{O}(h^p).
\end{align*}
If the numerical quadrature formula is order $q$ accurate, then 
\begin{align*}
L_d^E(q_0, q_1;h)&=\int_0^h L(q_{01}(t),v_{01}(t))dt\\
&=\left[h\sum\nolimits_{i=1}^m b_i L(q_{01}(c_i h), v_{01}(c_i h))\right]+\mathcal{O}(h^{q+1})\\
&=\left[h\sum\nolimits_{i=1}^m b_i L(\tilde{q}_{01}(c_i h)+\mathcal{O}(h^{p}), \tilde{v}_{01}(c_i h)+\mathcal{O}(h^{p}))\right]+\mathcal{O}(h^{q+1})\\
&=\left[h\sum\nolimits_{i=1}^m b_i L(\tilde{q}_{01}(c_i h), \tilde{v}_{01}(c_i h))\right]+\mathcal{O}(h^{p+1})+\mathcal{O}(h^{q+1})\\
&=L_d(q_0, q_1;h)+\mathcal{O}(h^{p+1})+\mathcal{O}(h^{q+1})\\
&=L_d(q_0, q_1;h)+\mathcal{O}(h^{\min(p,q)+1}),
\end{align*}
where we used the quadrature approximation error, the error estimates on the shooting solution, and the assumption that $L$ is Lipschitz continuous.
\end{proof}
\begin{remark}
Notice that the numerical quadrature formula introduces an additional factor of $h$ which gives the requisite $\mathcal{O}(h^{p+1})$ local error estimate for the discrete Lagrangian. Any consistent numerical quadrature formula would introduce this factor of $h$, since the integral is over the interval $[0,h]$.
\end{remark}

\begin{remark}
It should be noted that the prolongation-collocation variational integrators introduced in \cite{LeSh2011} can be viewed as a special case of the shooting-based variational integrator. In particular, this involves the collocation method on the prolongation of the Euler--Lagrange vector field as the one-step method, and Euler--Maclaurin formula as the numerical quadrature method.
\end{remark}

\paragraph{More general quadrature formulas.}
While we have primarily discussed numerical quadrature formulas that only depend on the integrand, one could consider more general quadrature formulas that depend on derivatives of the integrand, such as Gauss--Hermite quadrature. This would require information about the second- and higher-derivatives of $q$, which can be obtained by considering the prolongation of the Euler--Lagrange vector field. It is easy to show that prolongations of the vector field can be used to express all higher-derivatives of $q$ in terms of $q$, $\dot{q}$, and derivatives of the Lagrangian $L$. Therefore, the shooting method would provide the necessary information $(q,\dot{q})$ at the quadrature points to deduce the higher-derivatives $(\ddot{q}, q^{(3)}, \ldots)$, which in turn would allow the use of more general quadrature formulas.

\paragraph{Symmetric shooting-based discrete Lagrangians.} We now show that self-adjoint one-step methods and symmetric quadrature formulas yield self-adjoint shooting-based discrete Lagrangians. This guarantees that the resulting variational integrator is symmetric, and hence has even order of accuracy.
\begin{proposition}
Given a self-adjoint one-step method $\Psi_h$, and a symmetric quadrature formula $(c_i+c_{n-i}=1,\, b_i=b_{n-i})$, the associated shooting-based discrete Lagrangian is self-adjoint.
\end{proposition}
\begin{proof} 
The discrete Lagrangian is given by,
\[L_d(q_0, q_1;h)=h\sum\nolimits_{i=0}^n b_i L(q^i, v^i),\]
where the sequence $(q^i, v^i)$ satisfies
\[
(q^{i+1},v^{i+1})=\Psi_{(c_{i+1}-c_i)h}(q^i, v^i),\qquad
 q^0=q_0,\qquad
 q^n=q_1.
\]
The adjoint discrete Lagrangian is given by,
\[L_d^*(q_0,q_1;h)=- L_d(q_1, q_0;-h)=-(-h)\sum\nolimits_{i=0}^n b_i L(\tilde{q}^i, \tilde{v}^i),\]
where the sequence $(\tilde{q}^i,\tilde{v}^i)$ satisfies
\[
(\tilde{q}^{i+1},\tilde{v}^{i+1})=\Psi_{(c_{i+1}-c_i)(-h)}(\tilde{q}^i, \tilde{v}^i), \qquad
 \tilde{q}^0=q_1,\qquad
 \tilde{q}^n=q_0.
\]
Since $\Psi_h$ is self-adjoint, $\Psi_h=\Psi_h^*=\Psi_{-h}^{-1}$, and hence,
\[(\tilde{q}^i,\tilde{v}^i)=\Psi_{(c_{i+1}-c_i)h}(\tilde{q}^{i+1},\tilde{v}^{i+1}).\]
Also, $c_i+c_{n-i}=1$ implies that $(c_{i+1}-c_i)=(c_{n-i}-c_{n-i-1})$. From these two properties, it is easy to see that the two sequences are related by $q^{n-i}=\tilde{q}^i$, and $v^{n-i}=\tilde{v}^i$. Indeed, direct substitution allows us to rewrite $(\tilde{q}^i,\tilde{v}^i)=\Psi_{(c_{i+1}-c_i)h}(\tilde{q}^{i+1},\tilde{v}^{i+1})$ as $(q^{n-i},v^{n-i})=\Psi_{(c_{n-i}-c_{n-i-1})h}(q^{n-i-1},v^{n-i-1})$, which is equivalent to $(q^{i+1},v^{i+1})=\Psi_{(c_{i+1}-c_i)h}(q^i, v^i)$. Then,
\begin{align*}
L_d^*(q_0,q_1;h)&=-(-h)\sum\nolimits_{i=0}^n b_i L(\tilde{q}^i, \tilde{v}^i)=h\sum\nolimits_{i=0}^n b_i L(q^{n-i},v^{n-i})\\
&=h\sum\nolimits_{i=0}^n b_{n-i} L(q^i,v^i)=h\sum\nolimits_{i=0}^n b_i L(q^i,v^i)=L_d(q_0,q_1;h),
\end{align*}
which implies that the discrete Lagrangian is self-adjoint, where we used the fact that $b_i=b_{n-i}$.
\end{proof}

\paragraph{Group-invariant shooting-based discrete Lagrangians.} By the discrete Noether's theorem, a $G$-invariant discrete Lagrangian leads to a momentum-preserving variational integrator. In the Galerkin discrete Lagrangian construction, this can be achieved by choosing a $G$-equivariant interpolation space. We will see that $G$-equivariant one-step methods play a similar role in the construction of $G$-invariant shooting-based discrete Lagrangians. 

Let $\Phi:G\times Q\rightarrow Q$ be a group action of $G$ on $Q$, and let $T\Phi:G\times TQ\rightarrow TQ$ be the associated tangent-lifted action. A one-step method $\Psi_h:TQ\rightarrow TQ$ is $G$-equivariant if it commutes with $T\Phi$, i.e., $\Psi_h\circ T\Phi_g=T\Phi_g\circ \Psi_h$ for all $g\in G$.

\begin{proposition}
Given a $G$-equivariant one-step method $\Psi_h:TQ\rightarrow TQ$, and a $G$-invariant Lagrangian $L:TQ\rightarrow \mathbb{R}$, the associated shooting-based discrete Lagrangian is $G$-invariant.
\end{proposition}
\begin{proof}
For notational simplicity, we will denote the group action by $\Phi_g(q)=gq$, and the tangent-lifted action by $T\Phi_g(q,v)=(gq,gv)$. By definition, the shooting-based discrete Lagrangian is given by,
\[
L_d(q_0, q_1)=h\sum\nolimits_{i=0}^n b_i L(q^i,v^i),
\]
where the sequence $(q^i, v^i)$ satisfies
\[
(q^{i+1},v^{i+1})=\Psi_{(c_{i+1}-c_i)h}(q^i, v^i),\qquad
 q^0=q_0,\qquad
 q^n=q_1.
\]
Also,
\[
L_d(gq_0, gq_1)=h\sum\nolimits_{i=0}^n b_i L(\tilde{q}^i,\tilde{v}^i),
\]
where the sequence $(\tilde{q}^i, \tilde{v}^i)$ satisfies
\[
(\tilde{q}^{i+1},\tilde{v}^{i+1})=\Psi_{(c_{i+1}-c_i)h}(\tilde{q}^i, \tilde{v}^i),\qquad
 \tilde{q}^0=gq_0,\qquad
 \tilde{q}^n=gq_1.
\]
By $G$-equivariance of the one-step method $\Psi_h$, we have, 
\[
(gq^{i+1},gv^{i+1})=\Psi_{(c_{i+1}-c_i)h}(gq^i, gv^i),
\]
where $gq^0=gq_0$ and $gq^n=gq_1$. From this, we conclude that the two sequences are related by $(\tilde{q}^i,\tilde{v}^i)=(gq^i, gv^i)$. Hence,
\[
L_d(gq_0, gq_1)=h\sum\nolimits_{i=0}^n b_i L(gq^i,gv^i)=h\sum\nolimits_{i=0}^n b_i L(q^i,v^i)=L_d(q_0,q_1).
\]
where we used the $G$-invariance of the continuous Lagrangian.
\end{proof}
\subsection{Implementation issues}

While one can view the implicit definition of the discrete Lagrangian separately from the implicit discrete Euler--Lagrange equations,
\[ p_0 = -D_1 L_d(q_0, q_1;h),\qquad p_1=D_2 L_d(q_0,q_1;h),\]
in practice, one typically considers the two sets of equations together to implicitly define a one-step method:
\begin{subequations}
\begin{align}
L_d(q_0, q_1;h)&=h\sum\nolimits_{i=0}^n b_i L(q^i, v^i),\\
(q^{i+1},v^{i+1})&=\Psi_{(c_{i+1}-c_i)h}(q^i, v^i),\qquad i=0,\ldots n-1,\label{propagate}\\
q^0&=q_0,\label{bc0}\\
q^n&=q_1,\label{bc1}\\
p_0&= -D_1 L_d(q_0, q_1;h),\label{FL_d-}\\
p_1&=D_2 L_d(q_0,q_1;h)\label{FL_d+}.
\end{align}
\end{subequations}
\paragraph{Equation count.} The unknowns in the equations are $(q^i, v^i)_{i=0}^n$, $q_0$, $q_1$, and $p_0$, $p_1$. If the dimension of the configuration space $Q$ is $m$, then there are $2(n+3)m$ unknowns. Equations \eqref{bc0}-\eqref{FL_d+} yield $4m$ conditions, the propagation equations \eqref{propagate} give $2nm$ conditions, and the initial conditions $(q_0,p_0)$ give the last $2m$ conditions. While it is possible to solve this system directly using a nonlinear root finder, like Newton's method, another possibility is to adopt a shooting approach, which we will now describe.

\paragraph{Shooting-based implementation.} Given initial conditions $(q_0, p_0)$, we let $q^0=q_0$, and guess an initial velocity $v^0$. Using the propagation equations $(q^{i+1},v^{i+1})=\Psi_{(c_{i+1}-c_i)h}(q^i,v^i)$, we obtain  $(q^i, v^i)_{i=1}^n$. Then, we let $q_1=q^n$, and compute $p_1=D_2 L_d(q_0, q_1;h)$. However, unless the initial velocity $v^0$ is chosen correctly, the equation $p_0=-D_1 L_d(q_0,q_1;h)$ will not be satisfied, and one needs to compute the sensitivity of $-D_1 L_d(q_0,q_1;h)$ on $v^0$, and iterate on $v^0$ so that $p_0=-D_1 L_d(q_0,q_1;h)$ is satisfied. This gives a one-step method $(q_0,p_0)\mapsto(q_1,p_1)$. In practice, a good initial guess for $v^0$ can be obtained by inverting the continuous Legendre transformation $p=\partial L/\partial v$.

\subsection{Generalizations}
\paragraph{Hamiltonian Approach.} It might be preferable to express the variational integrator in terms of the Hamiltonian, and an underlying one-step method for Hamilton's equations. We first note that,
\[L(q,v)=\left.pv-H(q,p)\right|_{v=\partial H/\partial p}.\]
This leads to the following discrete Lagrangian,
\[ L_d(q_0,q_1;h) = h \sum\nolimits_{i=0}^n b_i \left[p^i v^i-H(q^i,p^i)\right]_{v^i=\partial H/\partial p (q^i, p^i)},\]
where 
\begin{align*}
(q^{i+1},p^{i+1})=\Psi_{(c_{i+1}-c_i)h}(q^i, p^i),\qquad
q^0=q_0,\qquad
q^n=q_1.
\end{align*}
While it is possible, as before, to solve this together with the implicit discrete Euler--Lagrange equations as a nonlinear system, it can also be solved using a shooting method.
\paragraph{Shooting-based implementation.} Given initial conditions $(q_0,p_0)$, we let $q^0=q_0$, guess an initial momentum $p^0$, and let $v^0= \partial H/\partial p (q^0, p^0)$. Using the propagation equations $(q^{i+1},p^{i+1})=\Psi_{(c_{i+1}-c_i)h}(q^i, p^i)$ and the continuous Legendre transformation $v^i=\partial H/\partial p (q^i, p^i)$, we obtain $(q^i, v^i, p^i)_{i=1}^n$. Then, we let $q_1=q^n$, and compute $p_1=D_2 L_d(q_0, q_1;h)$. As before, unless $p^0$ is correctly chosen, $p_0=-D_1 L_d(q_0,q_1;h)$ will not be satisfied, and one needs to compute the sensitivity of $-D_1 L_d(q_0,q_1;h)$ on $p^0$, and iterate on $p^0$ so that $p_0=-D_1 L_d(q_0,q_1;h)$ is satisfied. This gives a one-step method $(q_0,p_0)\mapsto(q_1,p_1)$.

\paragraph{Type II Hamiltonian Approach.} In some instances, for example, when the Hamiltonian is degenerate, the Type I generating function based approach to variational integrators using discrete Lagrangians is not applicable. A construction based on Type II generating functions, which corresponds to a discrete Hamiltonian~\cite{LeZh2009,LaWe2006}, can be adopted. This is based on the exact discrete Hamiltonian,
\[ H_d^{+,E}(q_0, p_1;h) = p(h) q(h) - \int_0^h\left[p(t)v(t)-H(q(t),p(t))\right]_{v=\partial H/\partial p}dt,\]
where $(q(t),p(t))$ is a solution of Hamilton's equations satisfying the boundary conditions $q(0)=q_0$, $p(h)=p_1$. A computable discrete Hamiltonian can be obtained by taking
\[ H_d^+(q_0,p_1;h) = p^n q^n - h \sum\nolimits_{i=0}^n b_i [ p^i v^i - H(q^i, p^i)]_{v^i=\partial H/\partial p (q^i, p^i)},\]
where 
\begin{align*}
(q^{i+1},p^{i+1})=\Psi_{(c_{i+1}-c_i)h}(q^i, p^i),\qquad
q^0=q_0,\qquad
p^n=p_1.
\end{align*}
The discrete Hamiltonian then yields a symplectic map via the discrete Hamilton's equations,
\begin{align*}
q_1=D_2 H_d^+(q_0,p_1),\qquad
p_0=D_1 H_d^+(q_0,p_1).
\end{align*}

\paragraph{Shooting-based implementation.} Given initial conditions $(q_0,p_0)$, we let $q^0=q_0$, guess an initial momentum $p^0$, and let $v^0= \partial H/\partial p (q^0, p^0)$. Using the propagation equations $(q^{i+1},p^{i+1})=\Psi_{(c_{i+1}-c_i)h}(q^i, p^i)$ and the continuous Legendre transformation $v^i=\partial H/\partial p (q^i, p^i)$, we obtain $(q^i, v^i, p^i)_{i=1}^n$. Then, we let $p_1=p^n$, compute $q_1=D_2 H_d^+(q_0, p_1)$, and iterate on $p^0$ until $p_0=D_1 H_d^+(q_0,p_1)$. 

\subsection{Examples}\label{section_examples}

\paragraph{Planar Pendulum.}
A planar pendulum of mass $m=1$ with the massless rod of length $l=1$ is a Hamiltonian system for which the equations of motion
are
$$
\dot{q}=p,\qquad \dot{p}=-\sin q.
$$
The total energy of the system is given by the Hamiltonian $H(q,p)=\frac{1}{2}p^2-\cos(q)$.

\begin{figure}[b]
\begin{center}
\subfigure[Global error vs. time step. The dotted line is the reference line for the exact order.]{\includegraphics[width=0.49\textwidth]{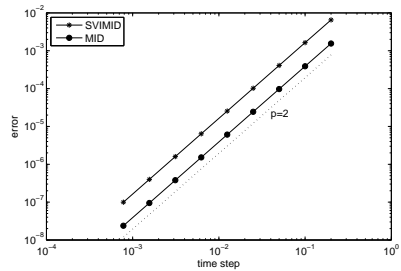}}
\subfigure[Global error vs. CPU time.]{\includegraphics[width=0.49\textwidth]{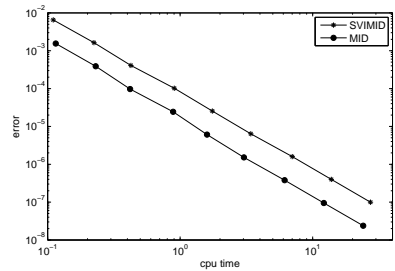}}
\subfigure[Energy error, step-size $h=0.2$]{\includegraphics[width=0.49\textwidth]{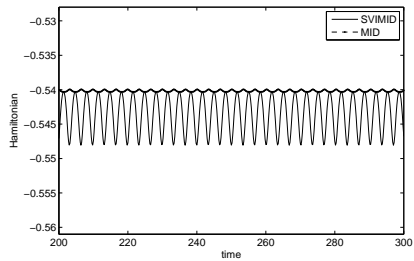}}
\vspace*{-2ex}
\caption{Pendulum. Comparison of a shooting-based variational integrator (SVIMID) and the Implicit Midpoint (MID).
}
\label{fig:genpen_error&cputimeMID}
\end{center}
\end{figure}

The plots in Figure~\ref{fig:genpen_error&cputimeMID},
represent the comparison of the standard implicit midpoint integrator, MID,
which is known to be symplectic and the shooting-based variational integrator, SVIMID. Following the construction outlined in Section~6, we employ 
the midpoint rule for the initial-value problem solved by shooting, combined with the trapezoidal quadrature rule to obtain a
computable discrete Lagrangian. The latter leads to a symmetric method, which is a $2$nd order symplectic integrator. 

\begin{figure}[t]
\begin{center}
\subfigure[Global error vs. time step. The dotted line is the reference line for the exact order.]{\includegraphics[width=0.49\textwidth]{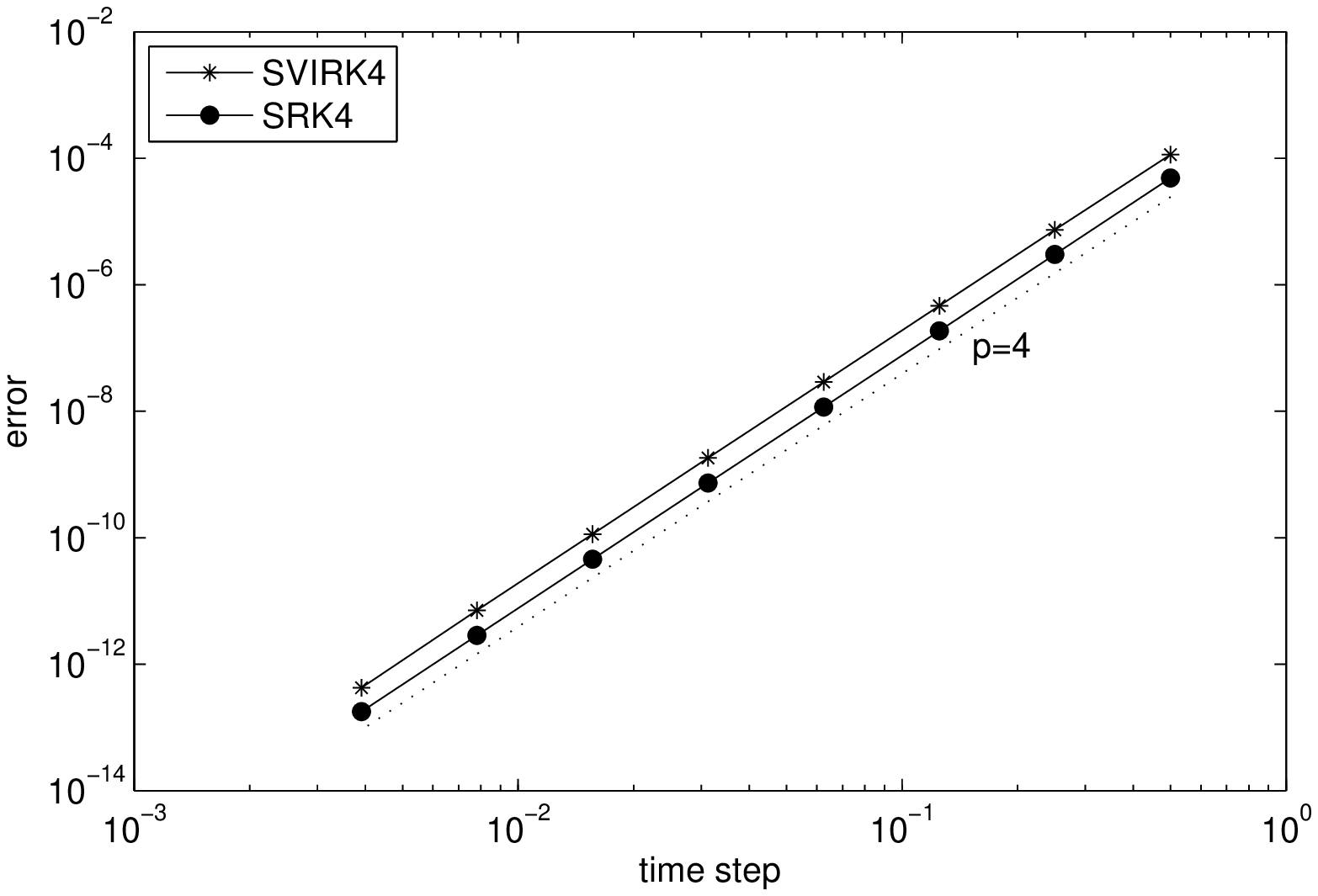}}
\subfigure[Global error vs. CPU time.]{\includegraphics[width=0.49\textwidth]{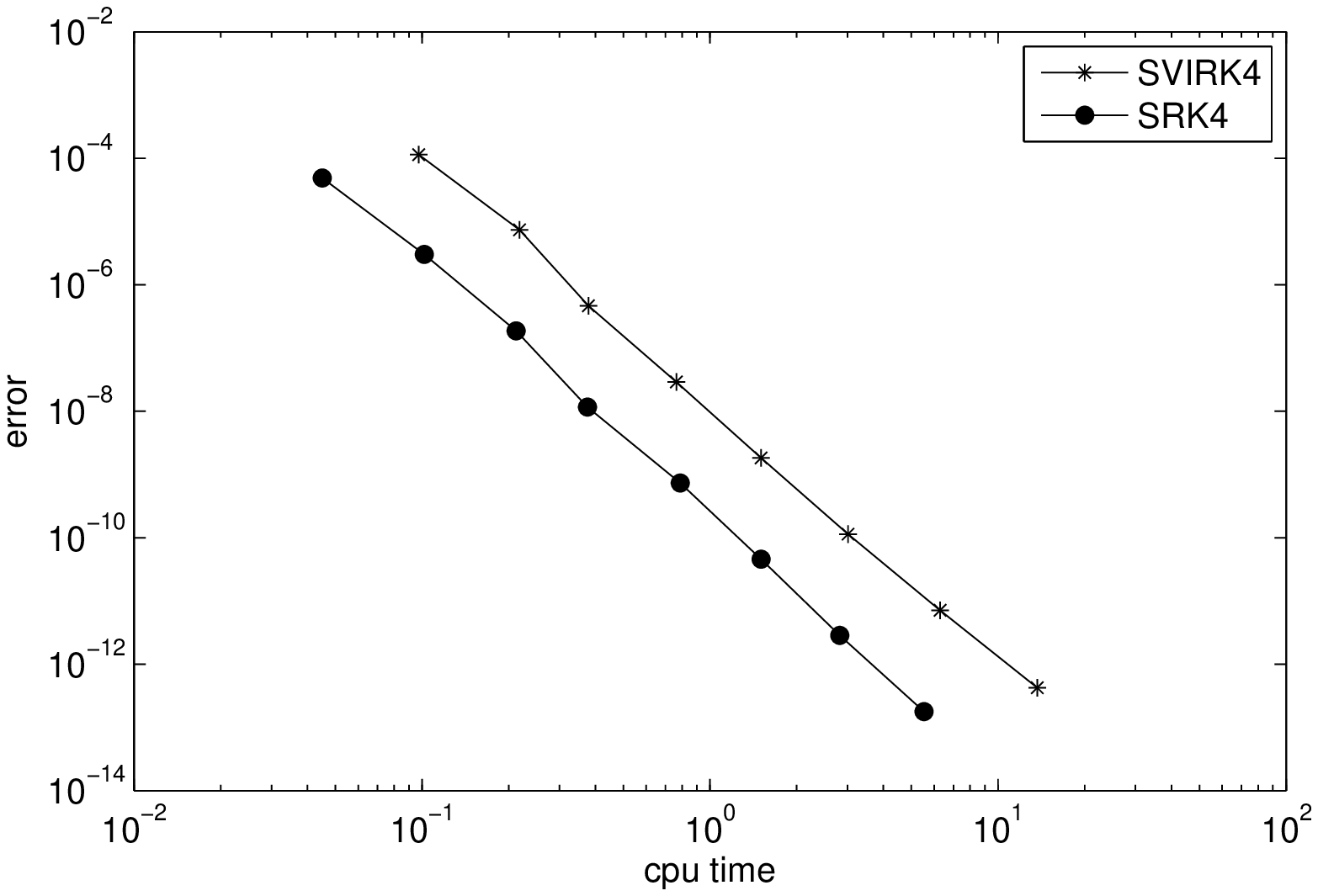}}
\subfigure[Energy error, step size $h=0.2$.]{\includegraphics[width=0.49\textwidth]{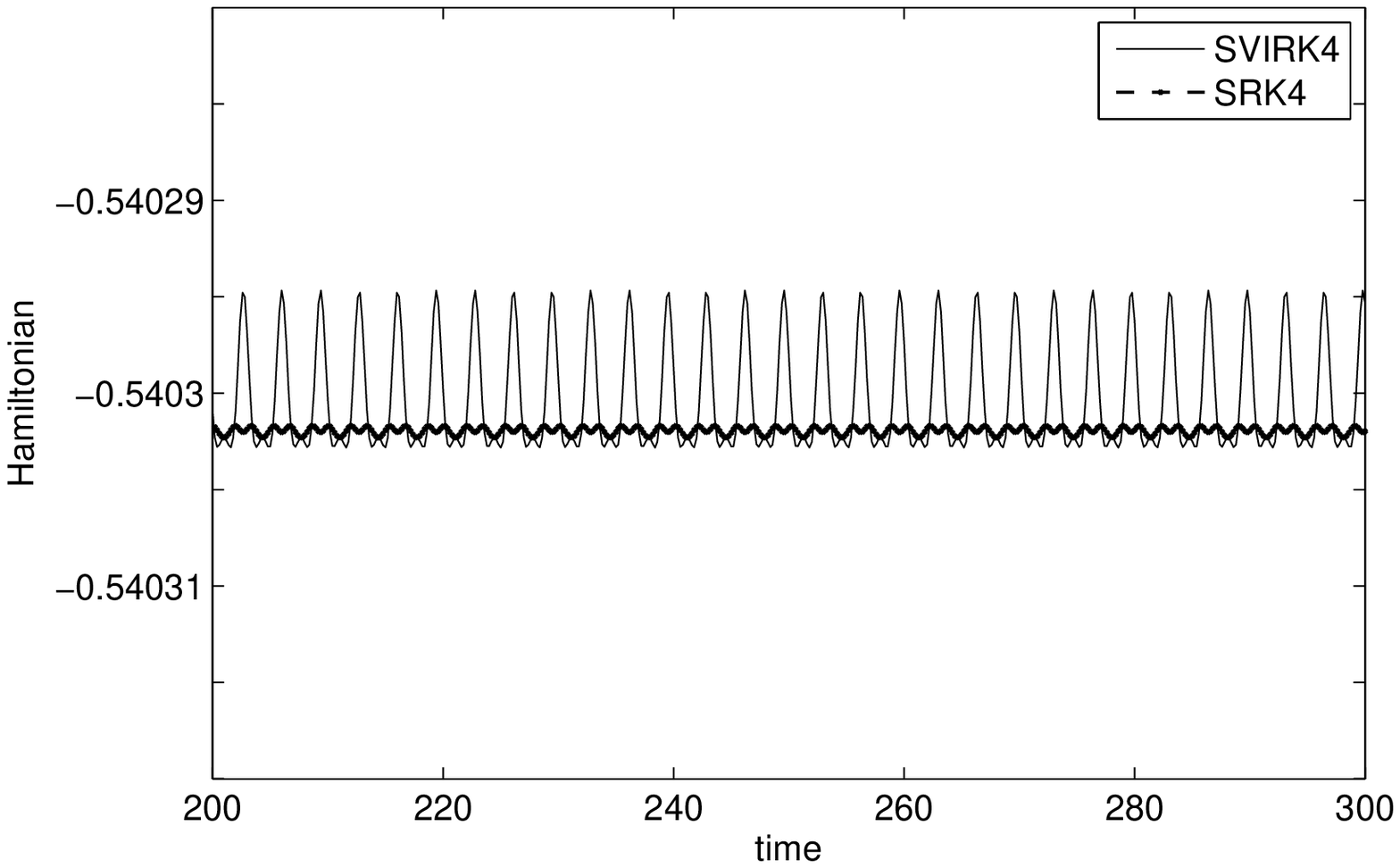}\label{fig:genpen_energyRK}}
\vspace*{-2ex}
\caption{Pendulum. Comparison of a shooting-based variational integrator (SVIRK4) and the symplectic Runge--Kutta (SRK4).}
\label{fig:genpen_error&cputimeRK}
\end{center}
\end{figure}

In Figure~\ref{fig:genpen_error&cputimeRK}, we compare the performance of the shooting-based variational integrator constructed from the explicit Runge--Kutta method and the Simpson quadrature rule with the two-stage
symplectic Runge--Kutta method. Both methods have global error of order $4$. The shooting-based method SVIRK4 exhibits a slightly larger error in energy (cf. Figure~\ref{fig:genpen_energyRK}), but it nevertheless stays bounded for large time-intervals.

\paragraph{Simple Harmonic Oscillator.}
We consider a harmonic oscillator system described by the equations
$$
\dot{q}=p,\qquad \dot{p}=-q.
$$
The total energy of the system is given by the Hamiltonian $H(q,p)=\frac{1}{2}p^2+\frac{1}{2}q^2$. 

In Figure~\ref{fig:simpen_error&cputime}, we compare two shooting-based variational integrators of order 4. The method sviRK4Sim, used above for the 
planar pendulum, involves the 4th order explicit Runge--Kutta integrator and the Simpson quadrature. The method sviRK4EM also employs the 4th order explicit Runge--Kutta method, but instead uses the Euler--Maclaurin quadrature, which involves higher-order derivatives of the integrant. In this case, the Euler--Maclaurin formula was truncated after the second term, and thereby requires the first-derivative of the Lagrangian function.

\begin{figure}[t]
\begin{center}
\subfigure[Global error vs. time step. The dotted line is the reference line for the exact order. ]{\includegraphics[width=0.49\textwidth]{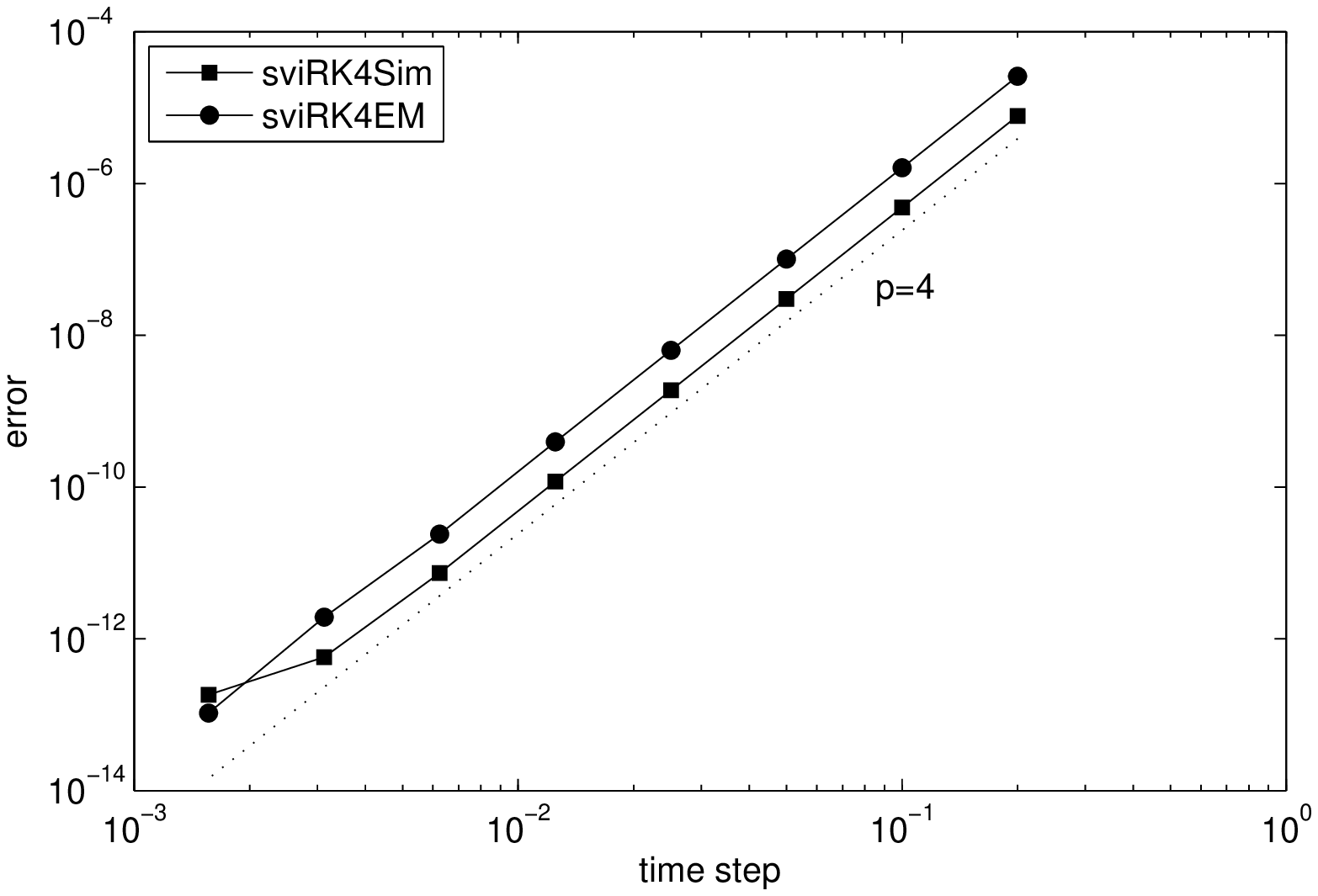}}
\subfigure[Global error vs. CPU time.]{\includegraphics[width=0.49\textwidth]{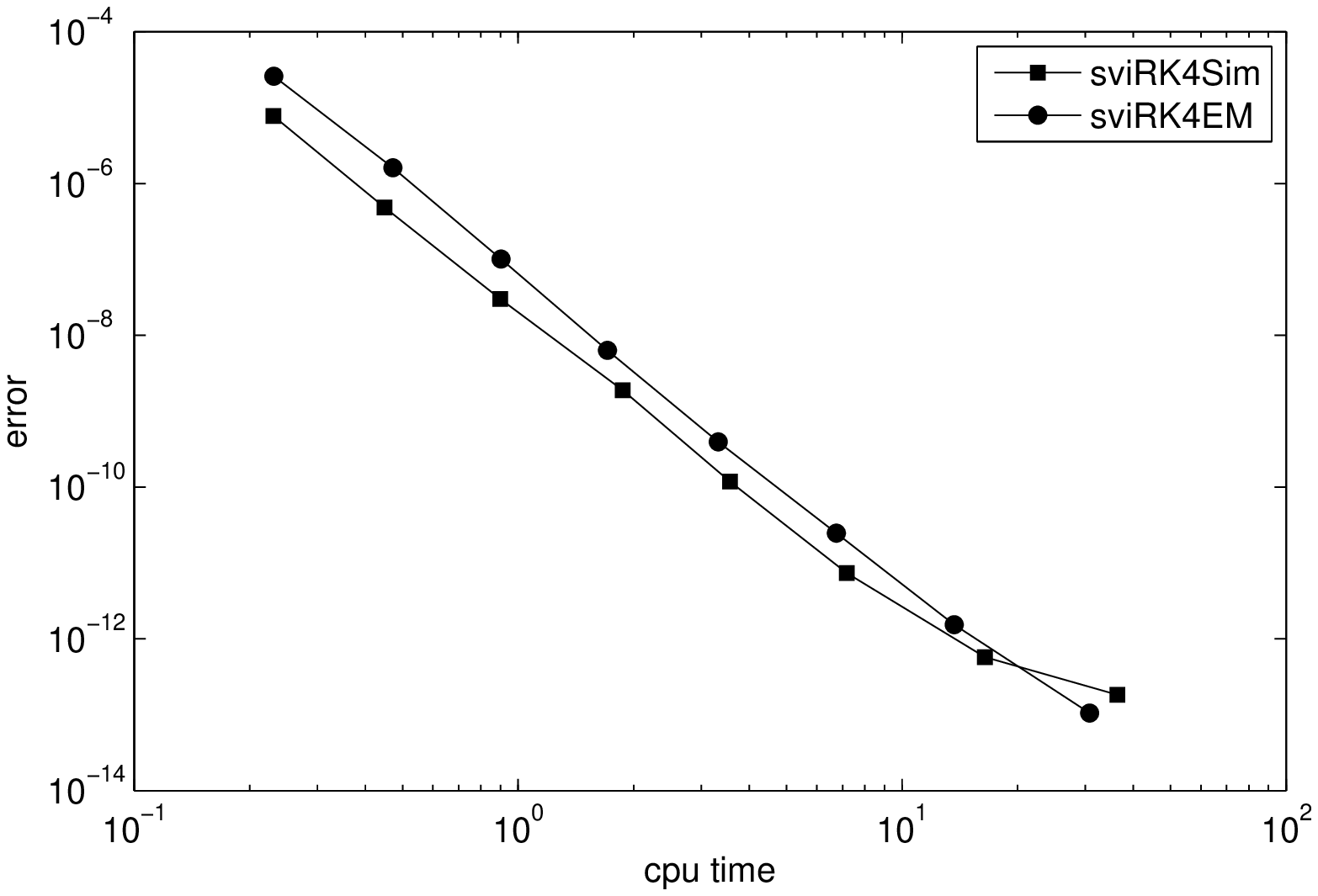}}
\subfigure[Energy error, step-size $h=0.2$.]{\includegraphics[width=0.49\textwidth]{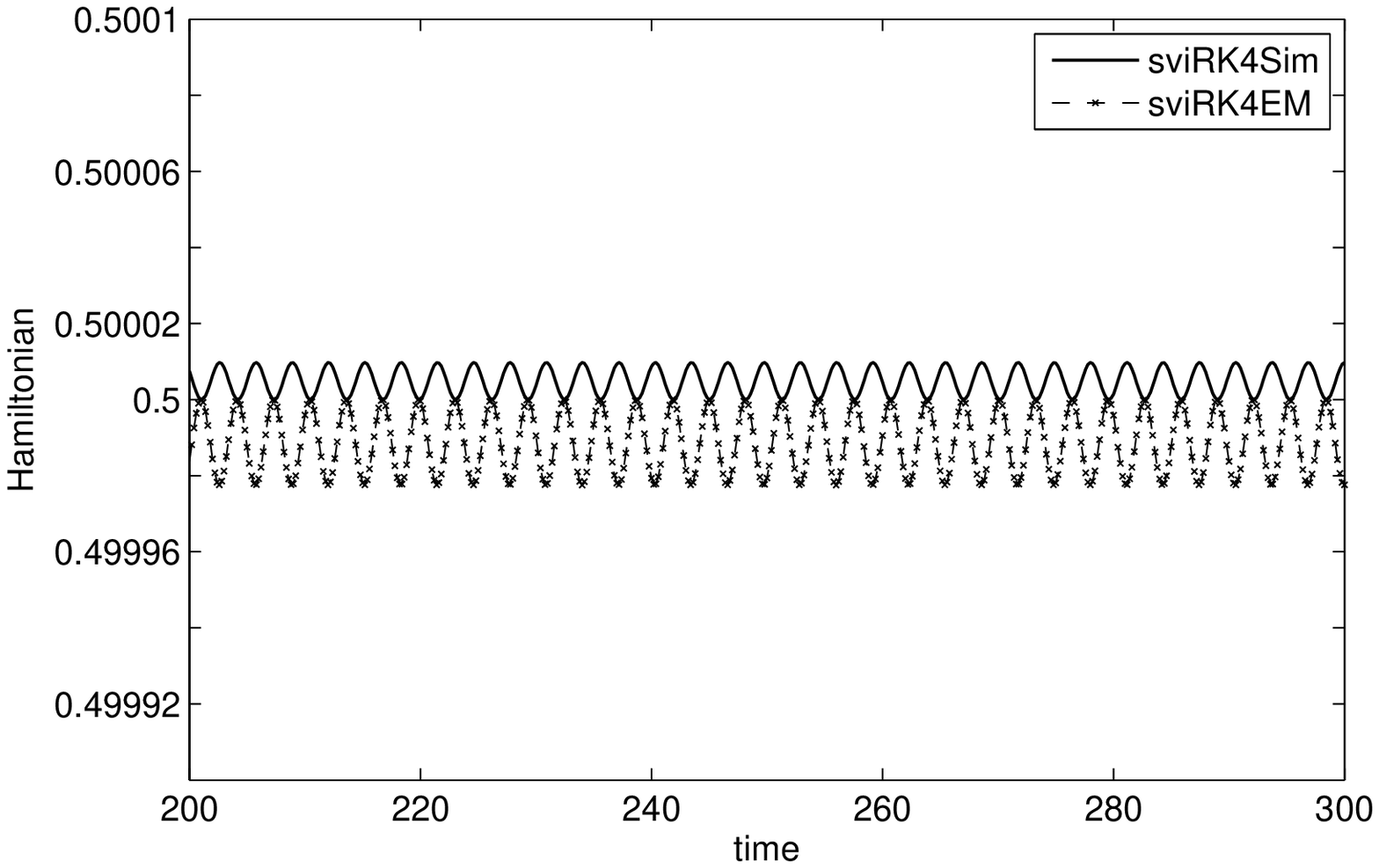}\label{fig:simpen_energy}}
\vspace*{-2ex}
\caption{Simple harmonic oscillator. Comparison of the variational integrator sviRK4Sim based on Simpson quadrature and sviRK4EM based on Euler--Maclaurin quadrature.}
\label{fig:simpen_error&cputime}
\end{center}
\end{figure}

The energy behavior for both sviRK4Sim and sviRK4EM is presented in Figure~\ref{fig:simpen_energy}. We note that both the global error and the energy error is generally better using Simpson's rule, as compared to the Euler--Maclaurin formula, even though the underlying one-step method is identical. This indicates that the specific choice of quadrature formula can have a significant effect on the global error and energy error properties of the constructed variational integrator.

\section{Conclusions} We presented two general techniques for constructing discrete Lagrangians: (i) the Galerkin approach, which depends on a choice of a quadrature formula, and a finite-dimensional function space; (ii) the shooting-based approach, which depends on a choice of a quadrature formula, and a one-step method.

The order of approximation and momentum-conservation properties of a variational integrator are related to the order of approximation and the group-invariance of the discrete Lagrangian, respectively. This results in a substantial simplification in the analysis of variational integrators, since it is easier to verify the approximation and group-invariance properties of the discrete Lagrangian than it is to directly verify the order of accuracy and momentum-conservation properties of the associated variational integrator.

For Galerkin variational integrators, the group-invariance of the discrete Lagrangian can further be reduced to the group-equivariance of the finite-dimensional function space. For shooting-based variational integrators, the order of the discrete Lagrangian is related to the order of the quadrature formula and one-step method, and the group-invariance of the discrete Lagrangian is related to the group-equivariance of the one-step method. Furthermore, the shooting-based implementation allows the variational integrator to partially inherit the computational efficiencies of the underlying one-step method. In particular, a shooting-based variational integrator constructed from an explicit one-step method will be more computationally efficient than one based on an implicit one-step method.

These two approaches provide an explicit link between the construction of variational integrators, approximation theory, and one-step methods for ordinary differential equations. In particular, this allows one to leverage existing theoretical results and techniques in approximation theory and the numerical analysis of time-integration methods in the construction and analysis of variational integrators. 

\section*{Acknowledgements} ML and TS were supported in part by NSF Grant DMS-1001521, and NSF CAREER Award DMS-1010687.

\bibliography{vi_review}
\bibliographystyle{spmpsci}

\end{document}